\documentclass[reqno,11pt]{amsart}

\usepackage{amsmath, amsthm, amssymb, amsfonts, amsthm,tikz}
\usepackage{amsxtra, amscd, geometry, graphicx,epic,eepic}
\usepackage[all,cmtip]{xy}
\usepackage{mathrsfs}
\usepackage{hyperref}
\usepackage{graphicx, nicefrac}
\usepackage{enumerate}
\usepackage{bbm}

\makeatletter
\newsavebox{\@brx}
\newcommand{\llangle}[1][]{\savebox{\@brx}{\(\m@th{#1\langle}\)}%
  \mathopen{\copy\@brx\kern-0.5\wd\@brx\usebox{\@brx}}}
\newcommand{\rrangle}[1][]{\savebox{\@brx}{\(\m@th{#1\rangle}\)}%
  \mathclose{\copy\@brx\kern-0.5\wd\@brx\usebox{\@brx}}}
\makeatother

\theoremstyle{definition}
\newtheorem{theorem}{Theorem}         
\newtheorem{proposition}[theorem]{Proposition}
\newtheorem{cor}[theorem]{Corollary}
\newtheorem{lemma}[theorem]{Lemma}

\newtheorem{definition}[theorem]{Definition}

\newtheorem{remark}[theorem]{Remark}
\newtheorem*{Ex1*}{Example 1}
\newtheorem*{Ex2*}{Example 2}
\newtheorem*{Ex3*}{Example 3}
\newtheorem*{Ex4*}{Example 4}

\newtheorem*{rmk*}{Remark}

\newcommand{\SSS}{\widetilde{\mathcal S}}
\newcommand{\Tr}{\text{Tr}}
\renewcommand{\S}{\mathcal{S}}
\newcommand{\W}{\mathcal{W}}
\newcommand{\A}{\mathcal{A}}
\newcommand{\N}{\mathcal{N}}
\newcommand{\R}{\mathcal{R}}

\numberwithin{equation}{section}

\title[Reduced quadratic irrationals arising from the OCF expansion]{Distribution of the reduced quadratic irrationals arising from the odd continued fraction expansion}
\author{Maria Siskaki}

\address{Department of Mathematics, University of Illinois at Urbana-Champaign,
Urbana, IL 61801}

\address{E-mail:  siskaki2@illinois.edu}

\begin{document}
             
\maketitle

\begin{abstract}
This paper investigates the quadratic irrationals that arise as periodic points of the Gauss type shift associated to the odd continued fraction expansion. It is shown that these numbers, which we call O-reduced, when ordered by the length of the associated closed primitive geodesic on some modular surface $\Gamma \backslash \mathbb{H}$, are equidistributed with respect to the Lebesgue absolutely continuous invariant probability measure of the Odd Gauss shift.
\end{abstract}

\section{Introduction}

According to Vall\'ee's classification in \cite{Vallee1} and \cite{Vallee2}, there are six classical Euclidean division algorithms. Their associated continued fraction (CF) expansions (see Table \ref{table vallee}) generate measure preserving transformations, known as Gauss shifts.

\begin{table}[h]\label{table vallee}
    \centering
    \begin{tabular}{|l|r|}
    \hline
    \textbf{Division Algorithm}     &  \textbf{CF expansion} \\
    \hline
      (G): By-default division   & Regular CF (RCF)\\
     \hline
    (B): By-excess division & Backwards CF (BCF)\\
     \hline
     (K): Centered division & Nearest Integer CF (NICF)\\
     \hline
    (E): Even division & Even CF (ECF) \\
     \hline
    (O): Odd division  & Odd CF (OCF) \\
     \hline
     (T): Subtractive algorithm & Lehner CF (Farey of RCF)\\
     \hline
     
    \end{tabular}
    \vspace{0.2cm}
    \caption{The Division Algorithms and their associated CF expansions (\cite{Vallee1}).}
\end{table}
The prototype example is the RCF. Every $x \in (0,1) \setminus \mathbb{Q}$ has a unique RCF expansion:
\begin{equation}\label{RCF expansion}
x=[a_1,a_2,\ldots]:=\cfrac{1}{a_1+\cfrac{1}{a_2+\cfrac{1}{\ddots}}}\, ,\qquad a_i \in \mathbb{N}:=\{1,2,3,\ldots\}.
\end{equation}

The RCF Gauss map $T$ is defined on $[0,1)$ by $T(0)=0$ and 
\begin{equation*}
    T(x)=  \frac{1}{x}-\bigg\lfloor \frac{1}{x}\bigg\rfloor, \quad  \text{if } x \neq 0.
\end{equation*}
The map $T$ acts on expansions of the form (\ref{RCF expansion}) as a shift: $T([a_1,a_2,a_3,\ldots])= [a_2,a_3,\ldots]$, and the digits of $x$ are completely recaptured by the iterates of $T$ as $a_1 =\lfloor 1/x\rfloor$,
$a_{n+1}=\lfloor 1/T^n (x)\rfloor$, $n\geq 1$. The probability
measure $\mu:=\frac{dx}{(1+x)\log 2}$ is $T$-invariant, and $T$ is exact. See \cite{KMS} for more on the ergodic properties of $T$.

Since $T$ is a shift, it is clear that the periodic points of $T$ are precisely the numbers with purely periodic RCF expansion. It is known \cite{Ono} 
that this set coincides with the set of reduced quadratic irrationals (QIs)
\begin{equation*}
 \{\omega \in [0,1) \mid \omega \text{ is a quadratic irrational with conjugate } \omega^* \in (-\infty,-1)\}.
\end{equation*}
The connection between $\omega$ and $\omega^*$ is given by the Galois formula (\cite{Ga})
\begin{equation*}
    [\, \overline{a_1, a_2,\ldots,a_n}\, ]^*= -\frac{1}{[ \, \overline{a_n,\ldots,a_2,a_1} \, ]},
\end{equation*}
where the notation $x=[\, \overline{a_1, a_2,\ldots,a_n} \, ]$ means that $x$ is periodic, and the block of digits $a_1, \ldots, a_n$ forms a (not necessarily least) period  of $x$.

The reduced QIs $\omega$ are naturally ordered by their length $\varrho(\omega)$, which measures the length of the closed primitive geodesic $\hat{\gamma}_\omega$ on the modular surface, obtained by projecting the geodesic $\gamma_\omega$ with endpoints at $\frac{1}{\omega}$ and $\frac{1}{\omega^*}$ on the hyperbolic half-plane $\mathbb{H}$. This geometric picture comes from \cite{Se}, where the two endpoints of  $\gamma_\omega$ are in $(1,\infty)$ and $(-1,0)$ respectively. For this reason, in this paper we will consider quadratic irrationals on $[1, \infty)$, even though we previously defined the RCF Gauss map and the set of reduced QIs  on $[0,1)$. All the definitions and results can be formulated in either of the two intervals $[0,1)$ or $[1, \infty)$ by conjugation by $x \mapsto \frac{1}{x}$.

A number theoretical approach initiated in \cite{KO}, further developed in \cite{Bo} and extended in \cite{Ustinov}, shows that the reduced QIs $\omega >1$, when ordered by their length $\varrho(\omega)$,  are equidistributed with respect to the Gauss measure $\mu$, together with an effective error term for the asymptotic formula: for every $\epsilon >0$ and $\alpha >1$ we have
\begin{align}\label{RCF equid}
    \bigg| \bigg\{ \omega >\alpha \  \Big| \ 
        \omega  \text{ reduced QI, }
        \varrho (\omega) \leq R 
     \bigg\} \bigg | & = \frac{3}{\pi^2} \ e^R  \int_\alpha^ \infty \frac{du}{u(u+1)} \ + \ \text{O}_{\epsilon} \big ( e^{(3/4 + \epsilon )R} \big),  
\end{align}
where $\frac{du}{u(u+1)}$ is the pushforward measure of $\mu$ on $[1,\infty)$. This result was proved on $[0,1]$ by Ustinov in \cite{Ustinov} and is reformulated on $[1,\infty)$ by conjugation by $x \mapsto \frac{1}{x}$, as mentioned above. 

This approach relies essentially on applications of the Weil bound for Kloosterman sums. The distribution of the reduced QIs had  previously been studied by Pollicott \cite{Po}, followed by Faivre \cite{Fa}, also more recently by Kelmer \cite{Ke} using Perron-Frobenius operators. However, their method involves the use of a Tauberian theorem, which does not allow for an effective error term in the asymptotic formula (\ref{RCF equid}). 

The distribution of the periodic points of the other Gauss shifts in Table \ref{table vallee} (or equivalently of certain types of reduced QIs) has been studied in the case of the ECF and the BCF expansions in \cite{BS}, and of the Lehner CF expansions in \cite{He}. The NICF and Nakada $\alpha$-expansions \cite{Nak, Kra1} are currently being investigated by the author.

Various problems related to OCF expansions have been studied in   \cite{BM, BV, HartKra, Kalp, Ri1, Ri2, Sc, Ustinov2}. In this paper we investigate the distribution of the O-reduced QIs, defined below. When ordered appropriately, they are equidistributed with respect to the invariant measure $\mu_o$ of the Odd Gauss shift $T_o$ (Corollary \ref{cor equidist of OCF QI}). Even more, we will show that the pairs $(\omega, -\omega^*)$ of O-reduced QIs are uniformly distributed with respect to the invariant measure $\widetilde{\mu}_o$ of the natural extension of $T_o$ (Theorem \ref{thm equidist of OCF QI}). 

One difference from the other types of QIs studied before 
is the range for $\omega^*$, which is obtained by  projecting the domain of the associated natural extension on the $y$-axis. Furthermore, the  matrices in the set $\SSS_e(\alpha, \beta;N)$ (cf. (\ref{defn of S_e})) that parametrizes the O-reduced QIs, need to satisfy additional conditions on the parity of the entries. In determining these matrices, the monotonicity properties of the O-convergents needs to be taken into account as well. All the above result in sums of different form in Section \ref{section computing asymptotic behavior}. Finally, Lemma \ref{bocaustinov} part (ii) provides an extension of Lemma 2 in \cite{Ustinov}, which is necessary because in this case the boundaries of the regions $\Omega_m(e, \alpha, \beta;N)$ (cf. (\ref{regions omega})) may be defined by lines of slope other than $\pm 1$ and $y$-intercept that is not an integer.

Every irrational number $x>1$ has a unique OCF-expansion 
\begin{equation}\label{OCFexpansion}
    x= [(a_1,e_1),(a_2,e_2),\ldots] = a_1+\cfrac{e_1}{a_2+\cfrac{e_2}{a_3+\ddots}},
\end{equation}
where  $a_i \in 2\mathbb{N}-1$, $e_i =\pm 1$, and $a_i+e_i \geq 2$, $\forall i \geq 1$. 

We denote $G= \frac{\sqrt{5}+1}{2}$, and $g= \frac{\sqrt{5}-1}{2}$.

\begin{definition}
A quadratic irrational (QI) $\omega>1$ is called \emph{O-reduced} if its conjugate $\omega^*$ satisfies $-G< \omega^* < 2-G$. Denote the set of O-reduced QI's by $\R_O$.
\end{definition}

In the notation of (\ref{OCFexpansion}), and similar to the RCF expansion, a Lagrange-type theorem (\cite{BM}, \cite{Daj4}, \cite{Pa}) characterizes $\R_O$ completely and provides the corresponding Galois formula for the conjugates of the O-reduced QIs:

\begin{proposition}\label{proposition BM}
A real number $\omega>1$ is an O-reduced QI if and only if it has a purely periodic OCF expansion. Moreover, if 
\begin{equation*}
    \omega =  [\,\overline{(a_1,e_1),(a_2,e_2),\ldots, (a_n,e_n)}\,],
\end{equation*}
 then 
\begin{equation*}
    \omega^* = - \llangle \, \overline{(a_n,e_n), \ldots , (a_1,e_1)} \,\rrangle ,
\end{equation*}
where 
\begin{equation}\label{grotesque definition}
    \llangle (a_1,e_1),(a_2,e_2), \ldots\rrangle := \cfrac{e_1}{a_1 + \cfrac{e_2}{a_2 + \ddots}}.
\end{equation}
  \end{proposition}
  
  The expansion (\ref{grotesque definition}) represents the grotesque continued fraction on $[G-2,G]$, introduced in \cite{Ri1} and \cite{Ri2}, and further studied in \cite{BM} and \cite{Sebe}.

To the end of ordering the O-reduced QIs, denote
\begin{align}\label{eq defn of A B matrices}
I & = \left( \begin{matrix} 1 & 0 \\ 0 & 1 \end{matrix}\right),  \quad A=  \left( \begin{matrix} 0 & 1 \\ 1 & 1 \end{matrix}\right), \quad B=  \left( \begin{matrix} 1 & 1 \\ 1 & 0 \end{matrix}\right), \quad \text{and } \\ 
\label{defn of Gamma}  \Gamma &= \{\sigma \in \text{PSL}(2,\mathbb{Z}) \mid \sigma \equiv I, A \text{ or } B \mod 2\}.
\end{align}

The closed primitive geodesics on the modular surface $\Gamma \backslash \mathbb{H}$ are precisely the geodesics that have a lift $\gamma_\omega$ to $\mathbb{H}$ with endpoints 

\begin{align*}
    \omega & = e[\, \overline{(a_1,e_1), (a_2,e_2), \ldots, (a_n,e_n)} \,], \text{ and } \\
    \omega^* & = -e \llangle \, \overline{(a_n,e_n), \ldots , (a_1,e_1)} \,\rrangle, \text{ for some } e = \pm 1,
\end{align*}
and such that $(-e_1)(-e_2) \cdots (-e_n)= +1$ (\cite{BM}). This is the OCF-analogue of the connection between the geodesics on PSL$(2,\mathbb{Z}) \backslash \mathbb{H}$ and the reduced QIs (\cite{Se}). 

Define the matrices

\begin{align}\label{Omega tilde defn}
 \nonumber   \Omega(\omega) & :=   \left( \begin{matrix} a_1 & e_1 \\ 1 & 0 \end{matrix}\right) \left( \begin{matrix} a_2 & e_2 \\ 1 & 0 \end{matrix}\right) \cdots \left( \begin{matrix} a_n & e_n \\ 1 & 0 \end{matrix}\right), \\
 \nonumber   &\\
    \widetilde{\Omega}(\omega) & := \begin{cases}
    \Omega(\omega), & \qquad  \text{if } \ (-e_1)(-e_2) \cdots (-e_n)= +1 \\
    \Omega(\omega)^2, & \qquad \text{if } \ (-e_1)(-e_2) \cdots (-e_n)= -1,
    \end{cases}
\end{align}
where $n$ is the least period of $\omega$.

Clearly, $\widetilde{\Omega}(\omega)$ has fixed points $\omega$ and $\omega^*$, and therefore preserves the geodesic $\gamma_\omega$ in $\mathbb{H}$.

Denote by $\mathfrak{r}(\sigma)$ the spectral radius of a $2\times 2$ matrix $\sigma$. It is well known that for $\sigma \in \text{SL}(2,\mathbb{Z})$ and $z \in \mathbb{H}$ on the axis of $\sigma$, the hyperbolic distance formula becomes

\begin{equation*}
    d(z, \sigma(z)) = 2 \log \mathfrak{r} (\sigma).
\end{equation*}

Thus, the closed primitive geodesic on $\Gamma \backslash \mathbb{H}$ with lift $\gamma_\omega$ has length $2 \log \mathfrak{r}(\widetilde{\Omega}(\omega))$. This induces a natural ordering of the O-reduced QIs:

\begin{definition}
The length of an O-reduced QI $\omega$ is defined as
\begin{equation*}
    \varrho_o(\omega) := 2 \log \mathfrak{r}(\widetilde{\Omega}(\omega)),
\end{equation*}
with $\widetilde{\Omega}(\omega)$ as in (\ref{Omega tilde defn}).
\end{definition}
Note that if $\omega $ is both a regular and an O-reduced QI, then it may be that   $\varrho_o(\omega) \neq \varrho(\omega)$, where $\varrho(\omega)$ is its length as a regular reduced QI. More details are given in Appendix 1.

We will prove the following estimate:
\begin{equation}\label{estimate}
    \big| \big\{ \omega \in  \R_O \  \big| \ \begin{array}{l}
        \varrho_o (\omega) \leq R 
    \end{array} \big\} \big |  = \frac{3 \log G}{4 \zeta(2)} e^R \ + \ O(e^{(\frac{3}{4}+ \epsilon)R}), \qquad \text{as } R \rightarrow \infty.
\end{equation}
This is a consequence of our main result:
\begin{theorem}\label{thm equidist of OCF QI}
 For every $\alpha  \geq 1$, $\beta_1 \geq G+1$, and $\beta_2 \geq G-1$ we have that 
 \begin{align*}
    \bigg| \bigg\{ \omega \in \R_O \  \Big| \
\substack{ \varrho_o (\omega) \leq R, \\  \omega \geq \alpha, \  -\frac{1}{\beta_1} \leq -\omega^* \leq \frac{1}{\beta_2}} 
       \bigg\} \bigg | & = \frac{e^R}{4\zeta(2)}  \log \bigg( \frac{\alpha \beta_1}{\alpha \beta_1 -1} \cdot  \frac{\alpha \beta_2 +1}{\alpha \beta_2} \bigg) + O(e^{(\frac{3}{4}+ \epsilon)R}) \\ 
    & = \frac{3 \log G }{4\zeta(2)}e^R \iint_{[\alpha,\infty) \times [-\frac{1}{\beta_1}, \frac{1}{\beta_2}]}  d\widetilde{\mu}_o + O(e^{(\frac{3}{4}+ \epsilon)R}),
\end{align*}
where $\widetilde{\mu}_o= \frac{1}{3 \log G} \cdot \frac{1}{(x+y)^2} \ dx dy$ is the invariant probability measure of the natural extension on $[1,\infty) \times [G-2,G]$ of the Odd Gauss map $T_o$, defined in Section \ref{section OCF}.
 \end{theorem}
 
 Taking $\beta_1 = G+1$, $\beta_2= G-1$ and applying (\ref{integral a to infintiy of mu o}), we get that the OCF reduced quadratic irrationals are equidistributed with respect to the invariant probability measure $\mu_o$ of $T_o$, given in (\ref{mu_o on 1 infinity}):
 \begin{cor}\label{cor equidist of OCF QI}
 For every $\alpha \geq 1$ we have that 
 \begin{align*}
    \bigg| \bigg\{ \omega \in \R_O \  \Big| 
          \varrho_o (\omega) \leq R, 
      \ \omega \geq \alpha  
       \bigg\} \bigg | & = \frac{3 \log G}{4 \zeta(2)} e^R \int_{\alpha}^\infty d\mu_o(x) \ + \ O(N^{(\frac{3}{4}+ \epsilon)R}).
 \end{align*}
 \end{cor}
Combining (\ref{estimate}) and Corollary \ref{cor equidist of OCF QI} we obtain
\begin{cor}\label{cor actual equidis of OCF QI}
For every $\alpha \geq 1$ we have that
\begin{equation*}
\lim\limits_{R \rightarrow \infty} \ \frac{{ \big|\big\{ \omega \in \R_O \  \big| \   \varrho_o (\omega) \leq R,\ \omega \geq \alpha \big\} \big |}}{  {\big| \big\{ \omega \in \R_O \  | \ \varrho_o (\omega) \leq R \big\} \big |}} = \int_{\alpha}^{\infty} d \mu_o.
\end{equation*}
\end{cor}

Taking $\alpha=1$ in Theorem \ref{thm equidist of OCF QI} we obtain
\begin{cor}\label{cor grotesque}
For every $\beta_1 \geq G+1$, $\beta_2 \geq G-1$ we have that
\begin{equation*}
    \bigg| \bigg\{ \omega \in \R_O \  \Big| 
     \begin{array}{ll}   {\color{white}aa}     \varrho_o (\omega) \leq R, &  \\
      -\frac{1}{\beta_1} \leq - \omega^* \leq     \frac{1}{\beta_2}&  \end{array}\bigg\} \bigg |= \frac{3 \log G}{4\zeta(2)} e^R \int_{-\frac{1}{\beta_1}}^{\frac{1}{\beta_2}} d \mu_G \ + \ O(e^{(\frac{3}{4}+ \epsilon)R}),
\end{equation*}
where $\mu_G(x)= \frac{1}{3 \log G} \frac{dx}{x+1}$. 
\end{cor}
Note that $\mu_G$ is the unique Lebesgue absolutely continuous probability invariant measure of the Gauss shift associated to the dual of the OCF expansion, called the grotesque continued fraction (\cite{Ri1}, \cite{Ri2}, \cite{Sebe}).

The organization of this paper is as follows: taking $N=\lfloor e^{R/2} \rfloor$ and replacing the inequality $\varrho_o(\omega) \leq R$ by $\Tr(\widetilde{\Omega}(\omega)) \leq N$,  we study the asymptotic behavior ($N \rightarrow \infty$) of the set 
\begin{equation}\label{defn of R_O}
    R_O(\alpha, \beta_1, \beta_2;N) = \Big\{ \omega \in \R_O \mid \omega \geq \alpha,\ -\frac{1}{\beta_2}< \omega^* < \frac{1}{\beta_1},\ \Tr(\widetilde{\Omega}(\omega)) \leq N \Big \},
\end{equation}
where $  \alpha >1, \ \beta_1 >G+1,\ \beta_2 > G-1$, and $\widetilde{\Omega}(\omega)$ as in (\ref{Omega tilde defn}). This set has the same asymptotic behavior as the set of the left hand side of  Theorem \ref{thm equidist of OCF QI} (cf. Lemma \ref{lemma replace spectral radius by trace}), i.e. replacing the spectral radius by the trace of $\widetilde{\Omega}(\omega)$ does not change the asymptotics.

In Section \ref{section OCF} we discuss the properties of the  OCF Gauss map system and  obtain the  $n$-th OCF-digit from the convergents of level $n$ and $n-1$ (Lemma \ref{digitfromconverg}). This information is necessary in order to parametrize the O-reduced QIs (cf. Remark \ref{remark why lemma is needed}). In  Subsection \ref{subsection 3.1} we reduce the problem of counting numbers $\omega$ with purely periodic OCF expansion to counting specific words $w$ which are in one-to-one correspondence with matrices $\sigma $ in a specific set $\mathcal{P}$ (cf. (\ref{defn of set P})). We then obtain a  characterization of $\mathcal{P}$ using the sets of matrices $\mathcal{S}_e$ (cf. Lemma \ref{lemma P= S+ union S-} and Lemma \ref{lemmaP=S}). Observe that counting the $\omega$'s with purely periodic OCF expansion  results in over-counting the $\omega$'s in $ \R_O$. In Subsection \ref{subsection 3.2} we reduce the problem of counting the $\omega$'s in  $R_O(\alpha, \beta_1, \beta_2;N)$ (cf. (\ref{defn of R_O})) to counting words $w$ in  $\W(\alpha, \beta_1, \beta_2;N)$ (cf. (\ref{defn of W})), showing that the over-counting does not affect the asymptotic behavior of the former as $N \rightarrow \infty$ (Remark \ref{remark over-counting}). In turn, instead of counting such $w$'s, we reduce the problem to counting matrices $\sigma$ in  sets of the form  $\widetilde{S}_e(\alpha, \beta;N)$ (cf. (\ref{defn of S_e}) and Lemma \ref{lemmaW=S+S}). We then show that counting such matrices $\sigma$ amounts to counting lattice points on modular hyperbolas that fall inside certain regions on the plane which depend on $N$ (Lemma \ref{Parametrization-}). In Section \ref{section computing asymptotic behavior} we count these lattice points and determine their asymptotic behavior as $N \rightarrow \infty$ to conclude the proof of Theorem \ref{thm equidist of OCF QI}. The difference between the (regular) reduced, the O- , E-, and the B- reduced QIs is illustrated in Appendix 1. The proof of Lemma \ref{bocaustinov} part (ii) is presented in Appendix 2. In Appendix 3 we prove some properties of the OCF convergents which will be used throughout.

\vspace{0.2cm}

\textbf{Acknowledgments. } The author would like to thank the referees, as their comments and suggestions significantly improved the presentation of this paper.

\section{Odd continued fractions}\label{section OCF}

Let $x>1$ and $a_1=a_1(x)$ be the (unique) odd positive integer such that $ x \in [a_1-1, a_1+1)$. Furthermore, if $x \geq a_1$, set $e_1=+1$, while if $x<a_1$ set $e_1=-1$. Now define the OCF Gauss map $T_o:(1,\infty)\rightarrow(1, \infty)$ as follows:
\begin{equation}\label{OCFGauss}
    T_o(x)= \frac{e_1}{x-a_1}.
\end{equation}
Setting $a_n=a_n(x)=a_1(T_o^{n-1}(x))$ and $e_n=e_n(x)=e_1(T_o^{n-1}(x))$, $n \geq 2$, we have that 
\begin{equation*}
    x=a_1+\cfrac{e_1}{a_2+\cfrac{e_2}{\cfrac{\ddots} {a_n+\cfrac{e_n}{T_o^{n}(x)}}}}.
\end{equation*}
In this way we get the OCF expansion of $x$ in (\ref{OCFexpansion}). The  map $T_o$ acts as a shift on the OCF digits:  $T_o([(a_1,e_1),(a_2,e_2),(a_3,e_3),\ldots]) = [(a_2,e_2),(a_3,e_3),\ldots]$. 

By Proposition \ref{proposition BM} it is obvious that a number $\omega>1$ is an O-reduced QI if and only if it is a periodic point of $T_o$. The probability measure 
\begin{equation}\label{mu_o on 1 infinity}
    \mu_o= \frac{1}{3 \log G} \cdot \bigg( \frac{1}{(G-1)x + 1}+ \frac{1}{(G+1)x-1} \bigg) \frac{dx}{x}
\end{equation}
is $T_o$-invariant. A simple calculation provides

\begin{equation}\label{integral a to infintiy of mu o}
    \int_{\alpha}^\infty d \mu_o(x) = \frac{1}{3 \log G} \log \bigg( \frac{G+1}{(G+1)\alpha -1} \cdot \frac{(G-1)\alpha +1}{G-1} \bigg).
\end{equation}

By conjugating $T_o$ by $x \mapsto \frac{1}{x}$, we get the customary OCF Gauss map on $[0,1]$ (see Figure \ref{Figure1}), whose properties were studied in \cite{Ri1, Ri2, Sc2, Sc}:

\begin{equation*}
    \overline{T}_o(x)= e\bigg( \frac{1}{x} -2k+1\bigg), \quad \text{where } x \in B(e,k):= \begin{cases} \big(\frac{1}{2k}, \frac{1}{2k-1} \big), & \text{if } e=1, \ k \geq 1 \\
    \big(\frac{1}{2k-1}, \frac{1}{2k-2} \big), & \text{if } e=-1, \ k \geq 2.
    \end{cases}
\end{equation*}
For the $\alpha$-OCF Gauss maps, defined similar to the maps in   \cite{Nak}, see \cite{BM2}.

The OCF-Gauss map $\overline{T}_o$ defined above has invariant probability measure 
\begin{equation*}
   \overline{\mu}_o= \frac{1}{3 \log G} \cdot \bigg( \frac{1}{G-1 + x}+ \frac{1}{G+1-x} \bigg) dx.
\end{equation*}

A natural extension of $T_o$ is given by the automorphism on $[1,\infty) \times [G-2,G]$ given by 
\begin{equation*}
    \widetilde{T}_o(u,v) = \bigg( T_o(u), \frac{e_1(u)}{a_1(u)+ v} \bigg) = \bigg( \frac{e_1(u)}{u-a_1(u)}, \frac{e_1(u)}{v+a_1(u)} \bigg),
\end{equation*}

with invariant probability measure $\widetilde{\mu}_o= \frac{1}{3 \log G} \cdot \frac{1}{(x+y)^2} \ dx dy$ (\cite{Sc}).

\begin{center}
\begin{figure}[h]
\includegraphics[scale=0.4]{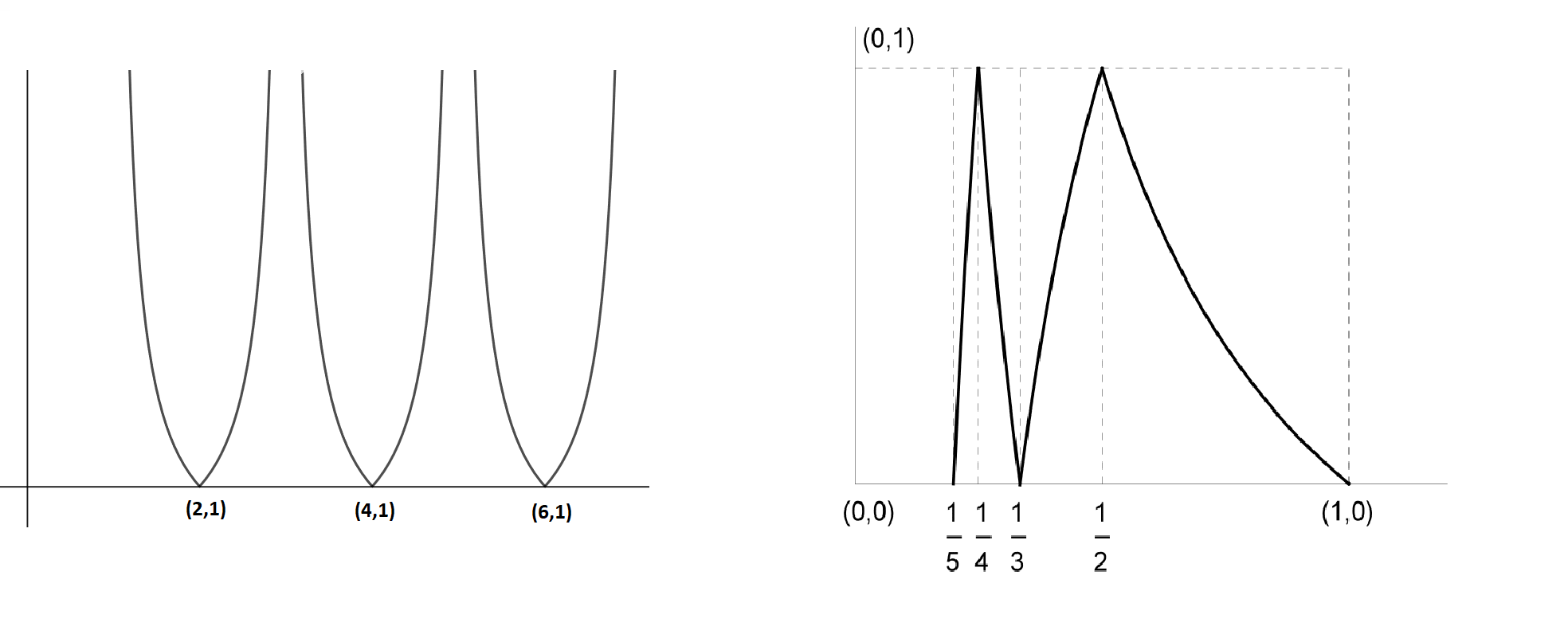} 
\caption{The graphs of the maps $T_o$ and $\overline{T}_o$. }\label{Figure1}
\end{figure}
\end{center}

The OCF convergents of $x$ satisfy
\begin{align}\label{recurrent relations}
  \nonumber   & p_0=1, \quad p_1=a_1, \quad p_n=a_np_{n-1} + e_{n-1}p_{n-2} \\
  \nonumber  & q_0=0, \quad q_1=1,   \quad q_n=a_nq_{n-1} + e_{n-1}q_{n-2},\\
  &  p_nq_{n-1}- q_np_{n-1} = (-e_0)(-e_1) \cdots(- e_{n-1}),  \text{ for all } n \geq 1, \text{ with } e_0=1.
\end{align}

 By (\ref{OCFGauss}) and the above recurrent relations for the convergents, we get that 
\begin{equation}\label{OCFnGauss}
    T^n_o(x)= 
    e_n  \frac{q_{n-1}x-p_{n-1}}{p_n-q_n x}.
\end{equation}

It is easy to see that $\frac{p}{q}$ is the $n$-th convergent of $x$ if and only if $\frac{q}{p}$ is the $n$-th convergent of $\frac{1}{x}$. 
\\

The recurrent relations show how to obtain the $n$-th convergents from the digit $a_n$, the sign $e_{n-1}$, and the previous convergents. The following Lemma shows how to obtain the digit $a_n$ and the sign $e_{n-1}$ from the convergents of level $n$ and $n-1$. Remark \ref{remark why lemma is needed} after Lemma \ref{lemmaP=S} below explains why this is needed to count the O-reduced QIs. 

\begin{lemma}\label{digitfromconverg}
Denote 
\begin{equation}\label{defn of sets F}
    \mathcal{F}_1  = \bigcup\limits_{k=0}^\infty (2k+g, \ 2k+1), \quad 
    \mathcal{F}_2  = \bigcup\limits_{k=1}^\infty [2k-1,2k), \quad \text{ and } \quad
    \mathcal{F}_3  = \bigcup\limits_{k=1}^\infty [2k, 2k+g).
\end{equation}

Let $x>1$, and $p_n,\ p_{n-1}$ be the numerators of the convergents of $x$ of level $n$ and $n-1$ respectively. For the digit $a_n$ and the sign $e_{n-1}$ of $x$ we have
\begin{equation*}
    a_n= \begin{cases} 
    \left \lfloor{\frac{p_n}{p_{n-1} }}\right \rfloor+1, & \text{ if }
p_n/p_{n-1} \in \mathcal{F}_1, \\[2ex]
\left \lfloor{\frac{p_n}{p_{n-1}}}\right \rfloor, & \text{ if } 
p_n/p_{n-1} \in \mathcal{F}_2, \\[2ex]
\left \lfloor{\frac{p_n}{p_{n-1} }}\right \rfloor-1, & \text{ if }
p_n/p_{n-1} \in \mathcal{F}_3,
    \end{cases}
    \quad \text{and } \quad e_{n-1}= \begin{cases}
     -1, & \text{ if } p_n/p_{n-1} \in \mathcal{F}_1, \\
    1, & \text{ if } p_n/p_{n-1} \in \mathcal{F}_2 \cup \mathcal{F}_3.
    \end{cases}
\end{equation*}

\end{lemma}
\begin{proof}By Lemma \ref{lemma bounds on ratio of convergents} in Appendix 3 we have that for all $n\geq 1$
\begin{equation}\label{claim}
    -g^2 < \frac{e_n p_{n-1}}{p_n}< G.
\end{equation}

Solving
\begin{align*}\label{eq from recurrent relations}
    \frac{p_n}{p_{n-1}}= a_n+\cfrac{e_{n-1}p_{n-1}}{p_{n-1}}
\end{align*}
for $a_n$ and using (\ref{claim}) we get that
\begin{equation}\label{digitinequal}
    \frac{p_n}{p_{n-1}}-G < a_n < \frac{p_n}{p_{n-1}} +2-G.
\end{equation}
Observe that this implies that $\frac{p_n}{p_{n-1}}>g$ for all $n \geq 1$. The interval $\mathcal{J}:=\big(\frac{p_n}{p_{n-1}}-G,\ \frac{p_n}{p_{n-1}} +2-G \big)$ has length 2 and therefore contains exactly two integers, one even and one odd. The digit $a_n$ is equal to the odd integer. 
\begin{itemize}
    \item If $\left \lfloor{\frac{p_n}{p_{n-1}}}\right \rfloor = 2k-1, \ k=1,2, \ldots$, then $1<G<2$ gives
    \begin{equation*}
         \frac{p_n}{p_{n-1}}-G < 2k-1 < \frac{p_n}{p_{n-1}} +2-G,
    \end{equation*}
    so that $a_n= 2k-1$.
    \item If $\left \lfloor{\frac{p_n}{p_{n-1}}}\right \rfloor = 2k, \ k=0, 1,2, \ldots$, we similarly get 
     \begin{equation*}
         \frac{p_n}{p_{n-1}}-G < 2k < \frac{p_n}{p_{n-1}} +2-G,
    \end{equation*}
    and the odd integer in $\mathcal{J}$ is either $2k-1$ or $2k+1$, depending on $\big\{\frac{p_n}{p_{n-1}}\big\}$:
    
    \begin{align*}
        2k-1 \in \mathcal{J} & \iff  \frac{p_n}{p_{n-1}}-G < 2k-1   \iff \bigg\{\frac{p_n}{p_{n-1}} \bigg \} <g \\
      2k+1 \in \mathcal{J} & \iff   2k+1< \frac{p_n}{p_{n-1}} +2-G  \iff \bigg \{\frac{p_n}{p_{n-1}} \bigg \}>g.
    \end{align*}
\end{itemize}
For the sign $e_{n-1}$, the  recursive relations for the convergents give 
\begin{equation}\label{signrecurs}
  \frac{p_n}{p_{n-1}}= a_n +  \frac{e_{n-1}}{p_{n-1}/p_{n-2}}.
\end{equation}
Combine this with (\ref{digitinequal}) to get  
\begin{equation*}
    G-2<\frac{e_{n-1}}{p_{n-1}/p_{n-2}}<G.
\end{equation*}
Therefore
\begin{equation*}
    \begin{split}
       & e_{n-1}=-1 \Leftrightarrow \frac{e_{n-1}}{p_{n-1}/p_{n-2}} \in (G-2,0) \Leftrightarrow  \frac{p_n}{p_{n-1}}- a_n \in (G-2,0) \Leftrightarrow \frac{p_n}{p_{n-1}} \in (a_n+G-2,a_n) \Leftrightarrow \\ 
       & \frac{p_n}{p_{n-1}} \in (2k-1+G-2,2k-1), \ k=1,2,\ldots \Leftrightarrow \frac{p_n}{p_{n-1}} \in (2k+g,2k-1), \ k=1,2,\ldots
    \end{split}
\end{equation*}
Similarly we get the case $e_{n-1}=1$.
\end{proof}

Combining the above with (\ref{signrecurs}), an easy calculation gives
\begin{equation}\label{ratiopn}
    \frac{p_{n-1}}{p_{n-2}} \in  
    \begin{cases}
    (2+g, \infty), & \text{ if } p_{n}/p_{n-1} \in \mathcal{F}_1\\
    (1, \infty), & \text{ if } p_{n}/p_{n-1} \in \mathcal{F}_2\\
    (g, 1], & \text{ if } p_{n}/p_{n-1} \in \mathcal{F}_3.\\
    \end{cases}
\end{equation}
\section{Parametrization of the O-reduced quadratic irrationals}\label{section parametrization of QIs}

\subsection{The correspondence between the O-reduced QIs, words, and matrices}\label{subsection 3.1} \mbox{}

Consider the O-reduced QI $\omega = [\, \overline{(a_1,e_1),(a_2,e_2), \ldots, (a_{n_0},e_{n_0})} \,]$, where $n_0$ is the least period of $\omega$, with multiplicity $k \geq 1$. Each pair $(\omega,k)$ gives rise to a unique word 
\begin{equation}\label{eq omega_k}
   w_{(\omega,k)}=  (a_1,e_1)(a_2,e_2) \ldots (a_n,e_n), \qquad   n=kn_0,
\end{equation}
 on the alphabet $(2\mathbb{N}-1) \times \{\pm1\}$, with $a_i+e_i \geq 2$, $i=1,2,\ldots,n$. The words of this form are, in turn, into one-to-one correspondence with the matrices in the set 

\begin{equation}\label{defn of set P}
     \mathcal{P} :=  \bigg\{  \left( \begin{matrix} a_1 & e_1 \\ 1 & 0 \end{matrix}\right)  \left( \begin{matrix} a_2 & e_2 \\ 1 & 0 \end{matrix}\right) \cdots \left( \begin{matrix} a_n & e_n \\ 1 & 0 \end{matrix}\right) \Big | \ n \geq 1, \ a_i \in 2\mathbb{N}-1, \ e_{i} = \pm 1, \ a_i + e_i \geq 2 \bigg\},
\end{equation}
via the map 
\begin{equation}\label{correspondence of words and matrices}
   w= (a_1,e_1)(a_2,e_2) \ldots (a_n,e_n) \mapsto M(w):= \left( \begin{matrix} a_1 & e_1 \\ 1 & 0 \end{matrix}\right)  \left( \begin{matrix} a_2 & e_2 \\ 1 & 0 \end{matrix}\right) \cdots \left( \begin{matrix} a_n & e_n \\ 1 & 0 \end{matrix}\right).
\end{equation} 
Observe that $\widetilde{\Omega}(\omega) \in \mathcal{P}$ (cf. (\ref{Omega tilde defn})). Now, the recurrence relations of the convergents give that

\begin{equation*}\label{product of building blocks}
    \left( \begin{matrix} a_1 & e_1 \\ 1 & 0 \end{matrix}\right)  \left( \begin{matrix} a_2 & e_2 \\ 1 & 0 \end{matrix}\right) \cdots \left( \begin{matrix} a_n & e_n \\ 1 & 0 \end{matrix}\right) = \left( \begin{matrix} p_n & e_n p_{n-1} \\ q_n & e_nq_{n-1} \end{matrix}\right).
\end{equation*}

This enables us to explicitly describe the set $\mathcal{P}$:

Denote
\begin{align}\label{eq building block}
 J & = \left( \begin{matrix} 0 & 1 \\ 1 & 0 \end{matrix}\right), \qquad
 M(a,e) =  \left( \begin{matrix} a & 1 \\ e & 0 \end{matrix}\right) = J\left( \begin{matrix} 0 & e \\ 1 & a \end{matrix}\right)J, \qquad \text{ and }\\
\nonumber &\\
\nonumber \widetilde{\Gamma} & = \{\sigma \in \text{GL}(2,\mathbb{Z}) \mid \sigma \equiv I, A \text{ or } B \mod 2\},
\end{align}
\text{ where $A$ and $B$ are defined in (\ref{eq defn of A B matrices}).}
Observe that $J \notin \widetilde{\Gamma}$, but $J\widetilde{\Gamma}J = \widetilde{\Gamma}$.

\begin{lemma}\label{lemma P= S+ union S-} Setting 
\begin{align*}
    \mathcal{S}_{+1} & =  \bigg\{ \left( \begin{matrix} a & b  \\ c &  d \end{matrix}\right) \in \widetilde{\Gamma} \mid 0 \leq d \leq b, \ 1 \leq c \leq a, \ a/b >g \bigg\} ,\\
     \mathcal{S}_{-1} & =  \bigg\{ \left( \begin{matrix} a & -b  \\ c &  -d \end{matrix}\right) \in \widetilde{\Gamma} \mid 0 \leq d \leq b, \ 1 \leq c \leq a, \ a/b > G+1 \bigg\},
\end{align*}
one has that 
\begin{equation*}
  \mathcal{P}   = \mathcal{S}_{+1} \cup \mathcal{S}_{-1}.
\end{equation*} 
\end{lemma}

\begin{proof}
Observe that for any matrix $ \left( \begin{smallmatrix} p_n & e_n p_{n-1} \\ q_n & e_nq_{n-1} \end{smallmatrix}\right) \in \mathcal{P}$ we have 
\begin{equation}\label{eq product of M(a,e)}
     \left( \begin{matrix} p_n & e_n p_{n-1} \\ q_n & e_nq_{n-1} \end{matrix}\right)  \left( \begin{matrix} 1 & 0 \\ 0 & e_n \end{matrix}\right) = \left( \begin{matrix} p_n & p_{n-1} \\ q_n & q_{n-1} \end{matrix}\right) =  M(a_1,e_0) M(a_2,e_1)\cdots \cdot M(a_n,e_{n-1}),
\end{equation}
where $M(a,e)$ as in (\ref{eq building block}) and $e_0=1$. The last equality above follows from the recurrent relations for the convergents. Therefore we will first parametrize the set 
\begin{align}\label{def of P tilde}
    \widetilde{\mathcal{P}} & = \big\{M(a_1,e_0) M(a_2,e_1)\cdots \cdot M(a_n,e_{n-1})  \mid n \geq 1,\ a_i \in 2\mathbb{N}-1,\ e_i = \pm 1,\ a_i+e_i \geq 2 \big\}. 
    \end{align}
This is done independently in Lemma \ref{lemmaP=S} below. Then combine Lemma \ref{lemmaP=S}   with (\ref{ratiopn}) to complete the proof.
\end{proof}

\begin{remark}\label{remark positive trace}
 Observe that the trace of any matrix in $\mathcal{P}$ is positive. Actually, we have  
 \begin{equation*}
     \Tr ( \widetilde{\Omega}(\omega)^k) \geq 3, \quad \text{ for every } \omega \in \R_O \text{ and every }k \geq 1,
 \end{equation*}
 with minimum attained for $\widetilde{\Omega}(\omega)= \left( \begin{smallmatrix} 1 & 1 \\ 1 & 0 \end{smallmatrix}\right)^2=\left( \begin{smallmatrix} 2 & 1 \\ 0 & 1 \end{smallmatrix}\right)$ and $\widetilde{\Omega}(\omega)= \left( \begin{smallmatrix} 3 & -1 \\ 1 & 0 \end{smallmatrix}\right)$. This shows that $\mathfrak{r}(\widetilde{\Omega}(\omega)) \geq 1+G$, for all $\omega \in \R_O$.\end{remark}


\begin{lemma}\label{lemmaP=S} Setting 
 \begin{align}\label{setSdefined}
   \mathcal{S} & = \Big\{\sigma = \left( \begin{matrix} a & b \\ c & d \end{matrix}\right) \in \widetilde{\Gamma} \mid 0
    \leq d \leq b, \ 1 \leq c \leq a ,\ a/b >g \Big\}, 
 \end{align}
 and with $ \widetilde{\mathcal{P}}$ as in  (\ref{def of P tilde}), we have that  $ \widetilde{\mathcal{P}} = \mathcal{S}$.
\end{lemma}

\begin{proof} To prove 
$\widetilde{\mathcal{P}} \subset \mathcal{S} $, let $M=M(a_1,e_0) M(a_2,e_1)\cdots  M(a_n,e_{n-1})=\left( \begin{matrix} p_n & p_{n-1} \\ q_n & q_{n-1} \end{matrix}\right) \in \widetilde{\mathcal{P}}$. Then det$M= \pm 1$ by (\ref{recurrent relations}) and $\frac{p_n}{p_{n-1}}>g$ by the proof of Lemma \ref{digitfromconverg}. 
By (\ref{eq p_n greater that q_n}) we have $p_n \geq q_n$, and since $M$ is a product of matrices  $ M(a_i,e_i) \equiv B \mod 2$ (cf. (\ref{eq defn of A B matrices})), and $B^2 \equiv A \mod 2$, $B^3 \equiv I \mod 2$, we have $M \equiv I, A \text{ or } B \mod 2$. Finally, take $x = [(a_1,e_1),\ldots, (a_n, e_n), (a_{n+1},e_{n+1}), \ldots]$. By (\ref{OCFnGauss}) and (\ref{eq product of M(a,e)}) we get that $M^{-1}(x)= e_n  T^n (x) \in (-\infty,-1)\cup(1,\infty)$.  Therefore $M \in \mathcal{S}$.
\\
To prove $\mathcal{S} \subset \widetilde{\mathcal{P}}$, let $\sigma = \left( \begin{matrix} a & b \\ c & d \end{matrix}\right) \in \mathcal{S}$.\\
Step 1: If min$\{a,b\}=1$, then the columns of $\sigma \in \mathcal{S}$ are successive convergents of a number $x>1$:
\begin{itemize}
    \item If $1=a \leq b$, then $\frac{a}{b} > g$ gives that $b =1$ and $1 \leq c \leq a$ gives that $c=1$. The determinant condition then gives that $d=0$, and thus $\sigma = B$ (cf. (\ref{eq defn of A B matrices})). Therefore its columns give the fractions $\frac{1}{0}$ and $\frac{1}{1}$, which are successive convergents for any $x=[(1,1),(a_2,e_2),\ldots]>1$. Here $\sigma$ is of the form $\left( \begin{matrix} a & 1 \\ 1 & 0 \end{matrix}\right)$, with $a \geq 1$ odd.  
    
    \item If $1=b<a$, then $d= 0$ or $1$. If $d=0$, then the determinant condition gives that $c=1$. Therefore $\sigma = \left( \begin{matrix} a & 1 \\ 1 & 0 \end{matrix}\right)$s with $a \geq 3$ odd, since $\sigma \in  \widetilde{\Gamma}$. The columns of this matrix are successive convergents for any $x=[(a,e_1), (a_2,e_2), \ldots]>1$. If $d=1$, then the determinant condition gives $0 \leq a-c=1$, so $\sigma = \left( \begin{matrix} c+1 & 1 \\ c & 1 \end{matrix}\right)$, $c$ odd because of the congruence conditions. The columns of this matrix are successive convergents for $x=[(1,1), (c,e_2),(a_3,e_3), \ldots]>1$. Observe that $\left( \begin{matrix} c+1 & 1 \\ c & 1 \end{matrix}\right)= \left( \begin{matrix} 1 & 1 \\ 1 & 0 \end{matrix}\right)  \left( \begin{matrix} c & 1 \\ 1 & 0 \end{matrix}\right)$, so that $\sigma $ is either of the form $\left( \begin{matrix} a & 1 \\ 1 & 0 \end{matrix}\right)$ (when $d=0$), with $a \geq 3$ odd, or of the form $\left( \begin{matrix} a & 1 \\ 1 & 0 \end{matrix}\right)  \left( \begin{matrix} c & 1 \\ 1 & 0 \end{matrix}\right)$ (when $d=1$), with $a$ odd.
\end{itemize}

Step 2: If min$\{a,b\}>1$, then there is a matrix $M(k_{\sigma}, e_{\sigma})=\left( \begin{matrix} k_{\sigma} & 1 \\ e_{\sigma} & 0 \end{matrix}\right)$ such that $k_{\sigma} \in 2\mathbb{N}-1$, $ \sigma_0 = \left( \begin{matrix} a_0 & b_0 \\ c_0 & d_0 \end{matrix}\right) :=\sigma  M(k_{\sigma}, e_{\sigma})^{-1} \in \mathcal{S}$, \ $k_{\sigma_{0}}+ e_{\sigma}\geq 2$, and $\text{min}\{a_0,b_0\} \leq \text{min}\{a,b\}$. To see this, recall (\ref{defn of sets F}) and define

\begin{equation*}
    k_\sigma= \begin{cases} 
\left \lfloor{\frac{a}{b }}\right \rfloor+1, & \text{ if }
a/b \in \mathcal{F}_1, \\[1ex]
\left \lfloor{\frac{a}{b}}\right \rfloor, & \text{ if } 
a/b \in \mathcal{F}_2,\\[1ex]
\left \lfloor{\frac{a}{b }}\right \rfloor-1, & \text{ if }
a/b \in \mathcal{F}_3,
    \end{cases} \quad \text{and } \quad  e_\sigma= \begin{cases}
     -1, & \text{ if } a/b \in \mathcal{F}_1, \\
     1, & \text{ if } a/b \in \mathcal{F}_2 \cup \mathcal{F}_3 .
    \end{cases}
\end{equation*}
\\[2ex]
Then $\sigma_0 = \left( \begin{matrix} a_0 & b_0 \\ c_0 & d_0 \end{matrix}\right)= \left( \begin{matrix} b & ae_\sigma-e_\sigma k_\sigma b \\ d & ce_\sigma-e_\sigma k_\sigma d \end{matrix}\right)$. A simple calculation shows that
\begin{equation*}\label{b_0>0}
    b_0=  ae_\sigma-e_\sigma k_\sigma b = \begin{cases}
  ( 1- \{\frac{a}{b}\}) \cdot b, & \text{ if } a/b \in \mathcal{F}_1 \\
  \{\frac{a}{b}\} \cdot b, & \text{ if } a/b \in \mathcal{F}_2\\
  (1+ \{\frac{a}{b}\}) \cdot b, & \text{ if } a/b \in \mathcal{F}_3
    \end{cases}\quad , \qquad   \text{ so that } b_0 > 0.
\end{equation*}
To show that $\sigma_0 \in \mathcal{S}$, first note that det$(\sigma_0) = \pm 1$. For the inequality $1\leq c_0 \leq a_0$, observe that $0 \leq c_0 \leq a_0$ by definition. If $c_0=d=0$, then the determinant condition for $\sigma$ gives that $b=c=1$, and in particular $\text{min}\{a,b\}=1$, a contradiction. Therefore $c_0= d \geq 1$. For the inequality $0 \leq d_0 \leq b_0$ consider the cases:
\begin{itemize}
    \item The case $d_0 <0$ contradicts the determinant condition for $\sigma_0$, since $a_0, b_0, c_0 > 0 $.
    \item If $d_0=0$, then the determinant condition for $\sigma_0$ combined with $c_0, b_0 \geq 0$ gives that $c_0=b_0=1$. Then $\sigma_0 = \left( \begin{matrix} b & 1 \\ 1 & 0 \end{matrix}\right)$, which satisfies $0 \leq d_0 \leq b_0$.
    \item If $d_0>0$, then the determinant condition for $\sigma_0$ gives 
    \begin{equation*}
        \frac{d_0}{b_0}= \frac{d}{b} \pm \frac{1}{bb_0},
    \end{equation*}
    so that $d_0/b_0 \leq 1$ is equivalent to $(b-d)b_0 \geq 1$, which is true whenever $b>d$. When $b=d$, the determinant condition for $\sigma$ gives that $b=d=1$, and that $\sigma = \left( \begin{matrix} c+1 & 1 \\ c & 1 \end{matrix}\right)= \left( \begin{matrix} 1 & 1 \\ 1 & 0 \end{matrix}\right) \left( \begin{matrix} c & 1 \\ 1 & 0 \end{matrix}\right) $. In particular, $0 \leq d_0 \leq b_0$ holds for $\sigma_0= \left( \begin{matrix} 1 & 1 \\ 1 & 0 \end{matrix}\right) $.
\end{itemize}
By the same calculation as in the proof of (\ref{ratiopn}), we get that $\frac{a_0}{b_0}>g$. Since $M(k_\sigma, e_\sigma) \equiv B \mod 2$, it follows that $\sigma_0 \equiv I, A $ or $B \mod 2$. To see that $k_{\sigma_0}+ e_\sigma \geq 2$, observe that by definition $e_\sigma =-1 \Leftrightarrow \frac{a}{b} \in \mathcal{F}_1$, and therefore (\ref{ratiopn}) gives that $\frac{a_0}{b_0}> 2+g$, so that $k_{\sigma_0} \geq 3$. To prove that min$\{a_0,b_0\} \leq \text{min}\{a,b\}$, consider the following cases:
\begin{itemize}
    \item If $\frac{a}{b} \in (g,1]$, then $0 < b_0< a<b $, so that min$\{a_0,b_0\} < \text{min}\{a,b\}$.
    \item If $\frac{a}{b}  \in (\mathcal{F}_1 \cup \mathcal{F}_2) \setminus (g,1]$, then $0<b_0<b<a$, so that min$\{a_0,b_0\} < \text{min}\{a,b\}$.
    \item If $\frac{a}{b}  \in \mathcal{F}_3$, then $b<a$ and $b_0 \geq b$, so that  min$\{a_0,b_0\} = \text{min}\{a,b\}$.
\end{itemize}
This concludes the proof that $\sigma_0 \in \mathcal{S}$.

Repeat this process successively on $\sigma_0,\sigma_1, \ldots$, until the minimum of the entries in the first row is equal to 1.  Observe that the process ends in finitely many steps: the situation $0 \leq \text{min}\{a_i,b_i\}=\text{min}\{a_{i-1}, b_{i-1}\}$ happens precisely when  $\frac{a_{i-1}}{b_{i-1}}  \in \mathcal{F}_3$. By (\ref{ratiopn}) it follows then that $\frac{a_i}{b_i}  \in (g,1]$, and therefore min$\{a_{i+1},b_{i+1}\} < $ min$\{a_{i}, b_{i}\}$. 

By Step 1, we get that $\sigma$ can be \emph{uniquely} expressed in the following way:
\begin{equation}\label{sigmaproduct}
 \sigma =  \left( \begin{matrix} k_{\sigma_n} & 1 \\ 1 & 0 \end{matrix}\right) \left( \begin{matrix} k_{\sigma_{n-1}} & 1 \\ e_{n-1} & 0 \end{matrix}\right) \cdots \left( \begin{matrix} k_{\sigma_0} & 1 \\ e_{\sigma_0} & 0 \end{matrix}\right)  \left( \begin{matrix} k_{\sigma} & 1 \\ e_\sigma & 0 \end{matrix}\right),
\end{equation}
where $ k_{\sigma_i} + e_{\sigma_{i-1}} \geq 2$, $\forall i =0,1,\ldots n$, and $\sigma_{-1} :=\sigma$. 

Step 3: For $x>1$, set  $t_{\sigma,x}:= \sigma^{-1}(x)$. Then $t_{\sigma,x} >1 $ or $t_{\sigma,x}<-1$ correspond to the situations $e_{\sigma+1}=1$ and $e_{\sigma+1}=-1$ respectively. Using $x= \sigma (t_{\sigma,x})$ and  (\ref{sigmaproduct}), we get that $x=[(k_{\sigma_n}, e_{n-1}), (k_{\sigma_{n-1}}, e_{n-2}), \ldots, (k_{\sigma_0}, e_{\sigma}), (k_{\sigma}+ t_{\sigma, x},*)]$, where $|t_{\sigma_x}|>1$. Therefore, the columns of $\sigma$ are successive convergents of $x$.  
\end{proof}

\begin{remark}\label{remark why lemma is needed}
 In Step 2 of the proof we reduce the matrix $\sigma$ ``one step down" to a matrix $\sigma_0 \in \mathcal{S}$ with smaller entries by multiplying $\sigma$ by the inverse of an appropriate matrix $M(k_\sigma, e_\sigma)$. 
 To determine the entries $k_\sigma$ and $e_\sigma$ of that matrix,  observe that, if what we were trying to prove was indeed true, then by (\ref{eq product of M(a,e)}) the matrices $\sigma$ and $\sigma_0$ should turn out to be of the form $\sigma = \left( \begin{matrix} p_n & p_{n-1} \\ q_n & q_{n-1} \end{matrix}\right)$, and $ \sigma_0=  \left( \begin{matrix} p_{n-1} & p_{n-2} \\ q_{n-1} & q_{n-2} \end{matrix}\right)$  for some $n$, with the $\frac{p_i}{q_i}$'s being the convergents of some $\omega$. Therefore, to pass from $\sigma$ to $\sigma_0$, we need to multiply $\sigma$ by the inverse of $M(k_\sigma, e_{\sigma})$, where $k_\sigma$ is the $n$-th digit and $e_{\sigma}$ is the $(n-1)$-th sign. To do that, we need to extract the $n$-th digit and the $(n-1)$-th sign from the columns of $\sigma$, i.e. from the convergents of level $n$ and $n-1$. This is done in Lemma \ref{digitfromconverg}.
\end{remark}

Observe that the matrices in $\mathcal{S}$ alone do not count the number of O-reduced QIs accurately. Indeed, discarding the information about $e_n$ has a significant effect on the number of words $w$ arising from an O-reduced QI because $e_n = -1$ implies $a_n \neq 1$.

\subsection{The correspondence between the O-reduced QIs in  $\R_O(\alpha,\beta;N)$, words and matrices}\label{subsection 3.2} \mbox{}

Consider an O-reduced QI $\omega=[\,\overline{(a_1,e_1), (a_2,e_2), \ldots ,(a_{n_0},e_{n_0})}\,]$, with least period $n_0$ and the corresponding word  (cf. (\ref{eq omega_k})) $w=w_{(\omega,k)}= (a_1,e_1)(a_2,e_2)\cdots(a_n, e_n)$, $n=kn_0$.  Define the trace of any such word $w$ as $\Tr(w):= \Tr (M(w))= p_n+ e_n q_{n-1}$, where $M(w)$ as in  (\ref{correspondence of words and matrices}). Clearly, when $w=w_{(\omega,k)}$ for some $\omega$ and some $k \geq 1$, then Tr$(w)=  \text{Tr}(\widetilde{\Omega}(\omega)^k)$. 

We order the O-reduced QIs $\omega$ according to the trace of their associated word, i.e. the trace of the associated matrix $\widetilde{\Omega}(\omega)$, instead of its spectral radius. For reasons explained in the introduction, we restrict our attention to $(-e_1)(-e_2)\ldots (-e_n)= +1$. 
This means that, in the notation of (\ref{setSdefined}), $\det(\sigma)= 1$.  It is easy to see that the error of replacing the spectral radius by the trace is negligible. We provide a short proof here for the convenience of the reader: 

\begin{lemma}\label{lemma replace spectral radius by trace}
 For every $\omega \in \R_O$ and every $k \geq 1$, we have that 
 \begin{equation}\label{ineq mathfral <Tra}
     \mathfrak{r}(\widetilde{\Omega}(\omega)^k) < \Tr (\widetilde{\Omega}(\omega)^k) < \mathfrak{r}(\widetilde{\Omega}(\omega)^k) + 1.
 \end{equation}
 Therefore, for every $\alpha \geq 1$, $\beta_1 \geq G+1$ and $\beta_2 \geq G-1$, we have
 
 \begin{equation}\label{eq Tr=r}
       \bigg| \bigg\{ \omega \in \R_O \  \Big| \
\substack{ \Tr(\widetilde{\Omega}(\omega))  \leq N, \\  \omega \geq \alpha, \  -\frac{1}{\beta_2} \leq \omega^* \leq \frac{1}{\beta_1}}  \bigg\} \bigg |  =  \bigg| \bigg\{ \omega \in \R_O \  \Big| \
\substack{ \mathfrak{r}(\widetilde{\Omega}(\omega)) \leq N, \\  \omega \geq \alpha, \  -\frac{1}{\beta_2} \leq \omega^* \leq \frac{1}{\beta_1}}  \bigg\} \bigg |.
 \end{equation}
\end{lemma}

\begin{proof} 
 Let $\widetilde{\Omega}(\omega)= \left( \begin{matrix} a & b \\ c & d \end{matrix}\right)$. By Remark \ref{remark positive trace} we have that $\eta:= \mathfrak{r}(\widetilde{\Omega}(\omega)) \geq 1+G$. Then $\Tr(\widetilde{\Omega}(\omega)^k)=\eta^k + \eta^{-k}$ and (\ref{ineq mathfral <Tra}) follows by the inequalities 
\begin{equation*}
    \eta^k < \eta^k + \eta^{-k} \leq \eta^k + \eta^{-1} < \eta^k + 1.
\end{equation*}

Recall that in (\ref{defn of R_O}) we denoted the set that appears on the left hand side of (\ref{eq Tr=r}) by $\R_O(\alpha, \beta_1, \beta_2;N)$. Denote the set  on the right hand side of (\ref{eq Tr=r}) by $\R_{O,\mathfrak{r}}(\alpha, \beta_1, \beta_2;N)$. By (\ref{ineq mathfral <Tra}) we have that 
\begin{equation*}
    \R_O(\alpha, \beta_1, \beta_2;N) \subset \R_{O,\mathfrak{r}}(\alpha, \beta_1, \beta_2;N) \subset \R_O(\alpha, \beta_1, \beta_2;N+1),
\end{equation*}
and thus 
\begin{equation*}
    | R_O(\alpha, \beta_1, \beta_2;N)| \leq |\R_{O,\mathfrak{r}}(\alpha, \beta_1, \beta_2;N)| \leq |\R_O(\alpha, \beta_1, \beta_2;N+1)|.
\end{equation*}
However, we have even more: suppose $\mathfrak{r}(\widetilde{\Omega}(\omega)) \leq N$. If  $N-1 < \mathfrak{r}(\widetilde{\Omega}(\omega)) \leq N$, then by (\ref{ineq mathfral <Tra}) and the fact that $\Tr(\widetilde{\Omega}(\omega))$ is a positive integer while $\mathfrak{r}(\widetilde{\Omega}(\omega))$ can not be an integer, we get $\Tr(\widetilde{\Omega}(\omega))=N$. If $\mathfrak{r}(\widetilde{\Omega}(\omega)) \leq N-1$, then by (\ref{ineq mathfral <Tra}) we have $\Tr(\widetilde{\Omega}(\omega)) < \mathfrak{r}(\widetilde{\Omega}(\omega)) +1 \leq (N-1)+1=N$. In conclusion,  we have $\R_{O,\mathfrak{r}}(\alpha, \beta_1, \beta_2;N) \subset \R_O(\alpha, \beta_1, \beta_2;N)$, and (\ref{eq Tr=r}) follows.
\end{proof}

\begin{remark}\label{remark finding eta}
In the notation of the above proof, it follows from straightforward calculations that for $\lambda_1:= c\omega+d$ and $\lambda_2:=\frac{1}{c\omega+d}$ we have $\lambda_1 \lambda_2= \det(\widetilde{\Omega}(\omega))$ and $\lambda_1 + \lambda_2 = \Tr(\widetilde{\Omega}(\omega)) \Leftrightarrow \omega$ is a fixed point of $\widetilde{\Omega}(\omega)$. Therefore, the two eigenvalues of $\widetilde{\Omega}(\omega)$ are $\lambda_1$ and $\lambda_2$.
\end{remark}

For any $\alpha>1$, $\beta_1 \geq G+1 $, $\beta_2 \geq G-1$, define the set
\begin{equation}\label{defn of W}
\begin{split}
  \W (\alpha, \beta_1, \beta_2; N) & := \Big\{w=(a_1,e_1) (a_2,e_2) \ldots (a_n,e_n) \mid n \geq 1,
 \ a_i \in 2\mathbb{N}-1, \ e_i = \pm 1, \\   & \hspace{0.9cm} a_i+e_i \geq 2 ,\ 
  (-e_1)(-e_2)\cdots(-e_{n})= 1 , \
 \Tr (w) \leq N, \\ 
 & \hspace{0.9cm} \omega := [\,\overline{(a_1,e_1), \ldots, (a_n,e_n)}\,] \geq \alpha, \ \omega^* \in \Big[-\frac{1}{\beta_2},\ \frac{1}{\beta_1}\Big] \Big \}.
\end{split}
\end{equation}

It is the asymptotics of this set which will provide the corresponding asymptotics for the O-reduced QIs.

\begin{remark}\label{remark over-counting}
Observe that if $w \in W (\alpha, \beta_1, \beta_2; N) $ and $\omega=\omega_w$ is the associated number, then the O-reduced QI $\omega$ is counted in $|\W (\alpha, \beta_1, \beta_2; N)|$  with multiplicity (cf. (\ref{eq omega_k}))
\begin{equation*}
    k= \max \{ k \mid \Tr(\widetilde{\Omega}(\omega)^k) \leq N \}.
\end{equation*} 
In other words, we have 
\begin{align*}
  | \W (\alpha, \beta_1, \beta_2; N)| &  = \sum\limits_{k \geq 1} |R_O(\alpha, \beta_1, \beta_2;N^{1/k})| 
  =|R_O(\alpha, \beta_1, \beta_2;N)|+  \ \sum\limits_{k \geq 2} |R_O(\alpha, \beta_1, \beta_2;N^{1/k})|,
\end{align*}
where $R_O(\alpha, \beta_1, \beta_2;N)$ as in (\ref{defn of R_O}). The terms corresponding to $k \geq 2$ do not have a significant effect asymptotically. Indeed, by Lemma  \ref{lemma replace spectral radius by trace} we have that $\eta^k < \Tr(\widetilde{\Omega}(\omega)^k) \leq N$, which gives that $k \leq \frac{\log N}{\log \eta} \leq \frac{\log N}{\log 2}$. Moreover, $\mathfrak{r}(\widetilde{\Omega}(\omega)^k) \leq N \Leftrightarrow \mathfrak{r}(\widetilde{\Omega}(\omega)) \leq N^{1/k}$, and there are $O(N^{2/k})$ O-reduced QIs $\omega$ such that the last inequality holds. Thus we get
\begin{align*}
   \sum\limits_{k \geq 2} |R_O(\alpha, \beta_1, \beta_2;N^{1/k})| &
   = O \bigg(  \sum\limits_{2 \leq k \leq \log_2 N} \lfloor (N+1)^{1/k}\rfloor^2 \bigg) = O \bigg(  \sum\limits_{2 \leq k \leq \log_2 N} N^{2/k} \bigg)
   = O(N \log N),
\end{align*} 
while the error term in the asymptotic behavior of $|\W (\alpha, \beta_1, \beta_2; N)|$ will turn out to be $O_\epsilon(N^{3/2+\epsilon})$ (cf. Lemma \ref{lemmaW=S+S}).
\end{remark}

Evidently, we need to consider cases depending on the last sign $e_n$:
\begin{equation*}
     \W_{e} (\alpha, \beta_1, \beta_2; N) := \{ w \in \W (\alpha, \beta_1, \beta_2; N) \mid e_n =e\},  \qquad e = \pm 1.
\end{equation*}
Using the Galois type formula for $\omega^*$ in Proposition \ref{proposition BM} we get that sign$(\omega^*)= -e_n$, and thus
\begin{align*}
    \W_{-1} (\alpha, \beta_1, \beta_2; N) &=  \W (\alpha, \beta_1, \beta_2; N) \cap \Big\{ w \mid 0< \omega^* < \frac{1}{\beta_1}\Big\}, \qquad  \text{ and }
    \end{align*}
    
    \begin{align*}
    \W_{+1} (\alpha, \beta_1, \beta_2; N) &=  \W (\alpha, \beta_1, \beta_2; N) \cap \Big\{ w \mid -\frac{1}{\beta_2}< \omega^* < 0 \Big \}. 
\end{align*}

We will next show that the number of words $w \in \W_e(\alpha, \beta_1, \beta_2; N)$, $e = \pm 1$, has the same asymptotics as $N \rightarrow \infty$ as the number of matrices $\sigma$ in explicitly defined sets of the form 
\begin{align}\label{defn of S_e}
     & \SSS_{e} (\alpha, \beta ; N) := \{\sigma \in \mathcal{S}_e \mid  \Tr(\sigma) \leq N, \ b/d \geq \alpha , \ a/b \geq \beta , \ ad-bc=  1 \}, \\
 \nonumber    &   \text{where }  e = \pm 1,  \ \beta \geq  \begin{cases} G+1,
     & \text{ if } e=-1, \\
   g, &  \text{ if } e=+1,  \end{cases}  \quad \text{and the sets } \S_e \text{ are defined as in Lemma \ref{lemma P= S+ union S-}}.
\end{align}

To do this, it is necessary to obtain to following results on the approximation of O-reduced QI's that give rise to words in $ \W_e(\alpha, \beta_1, \beta_2; N)$  by their convergents:
\begin{lemma}\label{lemma approximation by convergents}
Let $\omega= [\,\overline{(a_1,e_1), (a_2,e_2), \ldots (a_n,e_n)}\,]$ be such that it gives rise to a word $w \in  \W (\alpha, \beta_1, \beta_2; N)$ (cf. (\ref{defn of W})). Then 
\begin{equation}\label{eq approximation of omega by conv}
  \bigg|\omega - \frac{p_{n-1}}{q_{n-1}} \bigg| \leq \frac{1}{q_{n-1}q_n},  
\end{equation}
and 
\begin{equation}\label{absomegastar}
    \bigg| -\frac{1}{\omega^*}- \frac{e_n p_n}{p_{n-1}} \bigg| \leq \frac{G}{p_{n-1}(p_{n-1}- (2-G)q_{n-1})},
\end{equation}
where $\frac{p_k}{q_k}$ is the $k$-th convergent of $\omega$.
\end{lemma}

\begin{proof}
 For the first inequality, solve (\ref{OCFnGauss}) for $\omega$ and combine with $T_o^n(\omega)= \omega$ to get that $\omega= \frac{e_n p_{n-1}+ \omega p_n}{e_n q_{n-1}+ \omega q_n}$.   Therefore
\begin{equation}\label{Omegaconv}
  \omega - \frac{p_{n-1}}{q_{n-1}}  = \frac{\omega (p_nq_{n-1}-p_{n-1}q_n)}{q_{n-1} (\omega q_n + e_n q_{n-1})} = \frac{\omega e_n}{q_{n-1} (\omega q_n + e_n q_{n-1})}, 
\end{equation}
where we used (\ref{recurrent relations}) and (\ref{defn of W})  to get  $e_n(p_nq_{n-1}-p_{n-1}q_n)=(-e_1)\cdots (-e_n)=1$.
\begin{itemize}
    \item Case 1: $e_n=1$. Then 
    \begin{equation*}
         \bigg|\omega - \frac{p_{n-1}}{q_{n-1}} \bigg| = \frac{1}{q_{n-1}(q_n+\frac{1}{\omega} q_{n-1})} \leq \frac{1}{q_{n-1}q_n}.
    \end{equation*}
    \item Case 2: $e_n=-1$.  Then
    \begin{equation*}
         \bigg|\omega - \frac{p_{n-1}}{q_{n-1}} \bigg| =   \bigg|\frac{1}{q_{n-1}(q_n-\frac{1}{\omega} q_{n-1})}  \bigg|  = \frac{1}{q_{n-1} q_n |1-\frac{1}{\omega}\frac{q_{n-1}}{q_n}|}.
    \end{equation*}

    Combining $\omega >1$ with Lemma \ref{lemma bounds on ratio of convergents}, we get that  $1-\frac{1}{\omega}\frac{q_{n-1}}{q_n} \geq 1-\frac{q_{n-1}}{q_n} \geq 1-\frac{1}{G+1} =g$. Thus 
\begin{equation*}
    \bigg|\omega - \frac{p_{n-1}}{q_{n-1}} \bigg| \leq  \frac{g}{q_{n-1}q_n} \leq  \frac{1}{q_{n-1}q_n}.
\end{equation*}
\end{itemize}

For the second inequality, since $\omega^*$ is the conjugate of $\omega$, it is also a fixed point of the matrix $\left( \begin{matrix} e_n  q_{n-1} & -e_n  p_{n-1} \\ -q_n & p_n \end{matrix}\right)$ coming from (\ref{OCFnGauss}). Therefore
\begin{equation}\label{Omegastarconv}
    -\frac{1}{\omega^*}- \frac{e_n p_n}{p_{n-1}}= \frac{q_n \omega^* -p_n}{e_n q_{n-1}\omega^* - e_np_{n-1}} - \frac{e_n p_n}{p_{n-1}}= \frac{ \omega^*}{p_{n-1}(p_{n-1} - q_{n-1}\omega^*)},
\end{equation}
where we used again that $e_n=p_nq_{n-1}- q_np_{n-1}$ (cf. the equality  below (\ref{Omegaconv})).

Finally, use (\ref{Omegastarconv}), $-G \leq \omega^* \leq 2-G$, and   $p_{n-1} -q_{n-1}\omega^* >0 \Leftrightarrow \frac{p_{n-1}}{q_{n-1}}> \omega^*$, the latter being  true because $ \frac{p_{n-1}}{q_{n-1}} \geq 1 > 2-G> \omega^*$, to get (\ref{absomegastar}).
\end{proof}

We are now ready to prove that the asymptotic behavior of $\W(\alpha, \beta_1, \beta_2; N)$ is captured by the asymptotic behavior of the sets $\SSS_{e} (\alpha, \beta ; N)$ defined in (\ref{defn of S_e}). The error in the asymptotic formula is given (cf. the proof of Lemma \ref{lemmaW=S+S} below) by sets of the form 
\begin{align*}
    A_{N,1}(K) & = \bigg \{ \sigma = \left( \begin{matrix} a & eb \\ c & ed \end{matrix}\right)  \in \mathcal{S}_e \ \Big| \ b(b-(2-G)d) \leq G  N, \ e = \pm 1, \ a \leq K N \bigg\} \\
    A_{N,2}(K) & = \bigg\{ \sigma= \left( \begin{matrix} a & eb \\ c & ed \end{matrix}\right)   \in \mathcal{S}_e  \ \Big| \ cd\leq N , \ e= \pm 1 ,  \ a \leq K N \bigg\},
\end{align*}
where $K>0$ is a constant. The size of the error is computed independently in Lemma \ref{lemmaAasymptotics} to be $O_{\epsilon, K}(N^{3/2+\epsilon})$. Thus we get:

\begin{lemma}\label{lemmaW=S+S}
For every $\alpha \geq 1,\ \beta_1 \geq G+1,\ \beta_2 \geq G-1$, and for $\SSS_e(\alpha, \beta_i; N)$, $i=1,2$, as in (\ref{defn of S_e}), we have 
\begin{align*}
 |\W(\alpha, \beta_1, \beta_2; N)|  = |&\SSS_{-1} (\alpha, \beta_1 ; N)| + |\SSS_{+1} (\alpha, \beta_2 ; N)| \ + O_{\epsilon, \alpha, \beta_1}(N^{3/2+\epsilon}), \qquad (N \rightarrow \infty).
 \end{align*}
\end{lemma}

\begin{proof}
We first claim that 
\begin{align}\label{W-}
     |\W_{-1} (\alpha, \beta_1, \beta_2; N)|& \leq |\SSS_{-1} \big(\alpha,\beta_1 ; N\big) |.
\end{align}
 Let $w= (a_1,e_1) (a_2,e_2) \ldots (a_n,-1) \in \W_{-1} (\alpha, \beta_1, \beta_2; N)$, $\omega=[\,\overline{(a_1,e_1), (a_2,e_2), \ldots (a_n,-1)}\,]$, and $\sigma_w = \left( \begin{matrix} a & b \\ c & d \end{matrix}\right)=\left( \begin{matrix} p_n & -p_{n-1} \\ q_n & -q_{n-1} \end{matrix}\right)  \in \S_{-1}$ be the unique matrix associated to $w$ (cf. Lemma 9). To prove that $\frac{b}{d} \geq \alpha$, set $e_n=-1$ in (\ref{Omegaconv}) and observe that by the proof of Lemma \ref{lemma approximation by convergents}, Case 2 we have that $\omega q_n -q_{n-1}>0$. Thus $e_n=-1$  implies that  $\frac{p_{n-1}}{q_{n-1}}>\omega$, and $w \in \W_{-1} (\alpha, \beta_1, \beta_2; N) $ implies $  \omega \geq \alpha$, so that $\frac{b}{d}= \frac{p_{n-1}}{q_{n-1}} \geq \alpha$.
Also, setting $e_n=-1$ in equation (\ref{Omegastarconv}) and using  $p_{n-1}-q_{n-1}\omega^*>0$, we get that  $\frac{p_n}{p_{n-1}}> \frac{1}{\omega^*}$. Since $w \in \W_{-1} (\alpha, \beta_1, \beta_2; N) $ implies $0< \omega^*< \frac{1}{\beta_1}$, we get  $\frac{a}{b} = \frac{p_n}{p_{n-1}}> \beta_1$. Therefore, the matrix $\sigma_w$ belongs to $ \SSS_{-1} (\alpha,\beta_1 ; N)$ and the claim follows.

For the other direction, we claim that
\begin{align}\label{W--}
  |\SSS_{-1} (\alpha+ \frac{1}{N}, \beta_1 + \frac{1}{N}; N)|   
  &\leq  |\W_{-1} (\alpha, \beta_1, \beta_2; N)|+ |A_{N,1}(G)| + |A_{N,2}(G)|.
\end{align}
This suffices because the independently proved estimates (\ref{S+estimate}) and (\ref{S(-1)}) show that
\begin{equation*}
    |\widetilde{S}_e(\alpha \pm \frac{1}{N},\beta \pm \frac{1}{N};N)|-|\widetilde{S}_e(\alpha,\beta;N)|=O_\epsilon(N^{3/2+\epsilon}), \qquad e= \pm 1.
\end{equation*}


Take $\sigma = \left( \begin{matrix} a & -b \\ c & -d \end{matrix}\right) \in \SSS_{-1} (\alpha+ \frac{1}{N}, \beta_1 + \frac{1}{N} ; N) $, and let $w_{\sigma}=(a_1,e_1) (a_2,e_2) \ldots (a_n,-1)$  be the unique preimage of $\sigma $ through the map (\ref{correspondence of words and matrices}), and $\omega_{\sigma}= [\,\overline{(a_1,e_1), (a_2,e_2), \ldots, (a_n,-1)}\,]$ be the corresponding O--reduced QI. 

The trace condition $a \leq N + d$  together with  the inequalities  $(\alpha + \tfrac{1}{N}) d < b$ and $a > b (\beta_1 + \tfrac{1}{N})$ in the definition of $\SSS_{-1} (\alpha+ \frac{1}{N}, \beta_1 + \frac{1}{N}; N)$ (cf. (\ref{defn of S_e})) give that $d < \frac{b}{\alpha + 1 / N}< \frac{a}{(\alpha + 1/N)(\beta_1+ 1/N)}  \leq \frac{N+d}{(\alpha +1/N)(\beta_1 + 1/N)}$, so that $d \leq \frac{N}{\alpha \beta_1 -1}$, and $a \leq ( 1+\frac{1}{\alpha \beta_1 -1}) N$.  Observe that contrary to the case of the ECF \cite{BS}, here $a$ is uniformly bounded with respect to $N$: $a \leq ( 1+\frac{1}{\alpha \beta_1 -1})  N \leq ( 1+g ) N$, since $\alpha > 1$ and $\beta_1 > G+1$. This reflects the fact that the OCF invariant measure is finite, while the ECF invariant measure is infinite.

We will show that either $\sigma \in  A_{N,1}(G) \cup A_{N,2}(G)$, or $w_\sigma \in \W_{-1} (\alpha, \beta_1, \beta_2; N)$. Then the claim follows immediately.

Consider the following inequalities
\begin{align}\label{inequalities1}
     b(b-(2-G)d)> GN, 
     \end{align}
     \begin{align}\label{inequalities2}
     cd >N, 
     \end{align}
     
We first look at inequality (\ref{inequalities1}):
\begin{itemize}
    \item If it fails, then $\sigma \in A_{N,1}(G)$.
    \item If it holds, together with (\ref{absomegastar}) it gives that $\frac{1}{\omega_\sigma^*}$ is within $\frac{1}{N}$ distance of $\frac{p_n}{p_{n-1}}=\frac{a}{b} \geq \beta_1 +\frac{1}{N}$, where the last inequality holds because $ \sigma  \in \SSS_{-1} (\alpha+ \frac{1}{N}, \beta_1 + \frac{1}{N} ; N)$. 
     Thus $\frac{1}{\omega^*_\sigma} \geq \beta_1$, or  $\omega_\sigma^* \leq \frac{1}{\beta_1}$.
\end{itemize} 

We now look at inequality (\ref{inequalities2}):
\begin{itemize}
\item If it fails, then $ \sigma \in A_{N,2}(G)$.
    \item If it holds, together with (\ref{eq approximation of omega by conv}) it gives that $\omega_\sigma$ is within $\frac{1}{N}$ distance from $\frac{p_{n-1}}{q_{n-1}}= \frac{b}{d} \geq \alpha +\frac{1}{N}$, where the last inequality holds because $ \sigma  \in \SSS_{-1} (\alpha+ \frac{1}{N}, \beta_1 + \frac{1}{N} ; N)$. Thus $\omega_\sigma \geq \alpha$.
\end{itemize}

Finally, combining (\ref{recurrent relations}) and $ad-bc=1$ from the definition of $\SSS_{-1} (\alpha+ \frac{1}{N}, \beta_1 + \frac{1}{N} ; N)$ we get $(-e_1)\ldots(-e_n)=1$. Therefore, if both (\ref{inequalities1}) and (\ref{inequalities2}) hold, then $w_\sigma \in \W_{-1} (\alpha, \beta_1, \beta_2; N)$. This completes the proof of the claim.

Combining equations (\ref{W-}) and (\ref{W--}) with Lemma \ref{lemmaAasymptotics} below we get that
\begin{equation}\label{W-final}
    |\W_{-1} (\alpha, \beta_1, \beta_2; N)|= |\SSS_{-1} (\alpha, \beta_1 ; N)| + O_\epsilon(N^{3/2 + \epsilon}), \qquad (N\rightarrow \infty) .
\end{equation}

We now claim that 
\begin{align}\label{W+}
    |\W_{+1} (\alpha, \beta_1, \beta_2; N)| & \leq |\SSS_{+1} (\alpha-\frac{1}{N}, \beta_2-\frac{1}{N} ; N)| + |A_{N,1}(1)| + |A_{N,2}(1)|.
\end{align}

Let $w= (a_1,e_1) (a_2,e_2) \ldots (a_n,+1) \in \W_{+1} (\alpha, \beta_1, \beta_2; N)$, $\sigma_w = \left( \begin{matrix} a & b \\ c & d \end{matrix}\right)=\left( \begin{matrix} p_n & p_{n-1} \\ q_n & q_{n-1} \end{matrix}\right)  \in \SSS_{+1}$ be the unique matrix associated to $w$ (cf. Lemma \ref{lemma P= S+ union S-}), and $\omega_w=[\,\overline{(a_1,e_1), (a_2,e_2), \ldots ,(a_n,+1)}\,]$. Clearly, $\Tr(\sigma_w) \leq N \Rightarrow a \leq N$.

We first look at inequality (\ref{inequalities1}):
\begin{itemize}
    \item If it fails, then $\sigma_w \in A_{N,1}(1)$.
    \item If it holds, then by (\ref{absomegastar}) we have that $\frac{a}{b}= \frac{p_n}{p_{n-1}}$ is within $\frac{1}{N}$ distance from $-\frac{1}{\omega_w^*}> \beta_2$, where the last inequality holds because $w \in \W_{+1} (\alpha, \beta_1, \beta_2; N)$.  Thus $\frac{a}{b}> \beta_2-\frac{1}{N}$. 
\end{itemize}
We now look at (\ref{inequalities2}):
\begin{itemize}
    \item If it fails, then  $\sigma_w \in A_{N,2}(1)$.
    \item If it holds, together with (\ref{eq approximation of omega by conv}) it gives that $\frac{b}{d}= \frac{p_{n-1}}{q_{n-1}}$ is within $\frac{1}{N}$ from  $\omega_w \geq \alpha$, where the last inequality holds because $w \in \W_{+1} (\alpha, \beta_1, \beta_2; N)$. Thus $\frac{b}{d} \geq \alpha -\frac{1}{N}$.
\end{itemize}
 Additionally, we have $ad-bc=p_nq_{n-1}-p_{n-1}q_n= e_n=1$ (cf. equality below (\ref{Omegaconv}). Therefore, if both (\ref{inequalities1}) and (\ref{inequalities2}) hold, then $\sigma_w \in \SSS_{+1} (\alpha-\frac{1}{N}, \beta_2-\frac{1}{N} ; N)$. This completes the proof of the claim.

For the other direction, we claim that 
\begin{align}\label{W++}
 |\SSS_{+1} (\alpha+\frac{1}{N}, \beta_2 + \frac{1}{N}; N)| & \leq |\W_{+1} (\alpha, \beta_1, \beta_2; N)| + |A_{N,1}(1)| + |A_{N,2}(1)|.
\end{align}

Let $ \sigma \in \SSS_{+1} (\alpha, \beta_2 +\frac{1}{N}  ; N)$, $w_\sigma=(a_1,e_1) (a_2,e_2) \ldots (a_n,+1)$  be the unique word corresponding to $\sigma $ from the map (\ref{correspondence of words and matrices}), and $\omega_\sigma= [\,\overline{(a_1,e_1), (a_2,e_2), \ldots, (a_n,+1)}\,]$ be the corresponding O--reduced QI.

We first look at (\ref{inequalities1}):
\begin{itemize}
    \item If it fails, then $\sigma \in A_{N,1}(1)$.
    \item If it holds, together with (\ref{absomegastar}) we have that $-\frac{1}{\omega_\sigma^*}$ is within $\frac{1}{N}$ distance from $\frac{p_n}{p_{n-1}}= \frac{a}{b} \geq \beta_2+\frac{1}{N}$. Thus $-\frac{1}{\omega_\sigma^*} \geq  \beta_2 >0$, or $-\frac{1}{\beta_2} \leq \frac{1}{\omega_\sigma^*}<0$.
\end{itemize}

We then look at (\ref{inequalities2}):
\begin{itemize}
    \item If it fails, then $\sigma \in A_{N,2}(1)$.
    \item If it holds, together with (\ref{eq approximation of omega by conv}) we get that $\omega_w$ is within $\frac{1}{N}$ distance from $\frac{p_{n-1}}{q_{n-1}}= \frac{b}{d} \geq \alpha +\frac{1}{N}$, so that $\omega_w \geq \alpha$. 
\end{itemize}
Similar to the previous cases we have $(-e_1)\cdots(-e_n)=1$, so ultimately, if both (\ref{inequalities1}) and (\ref{inequalities2}) hold, we have that $w \in \W_{+1} (\alpha, \beta_1, \beta_2; N)$. This completes the proof of the claim. 

Combining equations (\ref{W+}) and (\ref{W++}) with Lemma \ref{lemmaAasymptotics}  below we get that
\begin{equation}\label{W+final}
     |\W_{+1} (\alpha, \beta_1, \beta_2; N)|= |\SSS_{+1} (\alpha, \beta_2 ; N)| + O_\epsilon(N^{3/2 + \epsilon}), \qquad (N \rightarrow \infty).
\end{equation}

Finally, relations (\ref{W-final}) and (\ref{W+final}) complete the proof.
\end{proof}

The size of the error (cf. above Lemma \ref{lemmaW=S+S}) is determined in the following Lemma. In what follows we use the notation $x \ll M$ to mean that there is a constant $C>0$ such that $|x| \leq CM$.
\begin{lemma}\label{lemmaAasymptotics}
For  $r=1,2$ and  $K>0$ we have
\begin{equation*}
    |A_{N,r}(K)|\ll_{\epsilon,K}N^{3/2 + \epsilon}.
\end{equation*}
\end{lemma}
\begin{proof}  In any of the cases $r=1,2$, $e= \pm 1$, it is obvious that we have $\leq K N$ choices for $a$.

For $\sigma = \left( \begin{matrix} a & e b \\ c & e d \end{matrix}\right) \in A_{N,1}(K)$, consider the following cases:
\begin{itemize}
\item If $b \leq \sqrt{N}$, then $d \leq b \leq \sqrt{N}$ (cf. the definition of $\mathcal{S}_e$ in the statement of Lemma \ref{lemma P= S+ union S-}), so there are at most $\sqrt{N}$ choices for $d$. Fix  $a$ and $d$, and use the determinant condition to get $c b = ad \pm 1 \leq K N^{3/2} \pm 1$. Therefore, the possible choices for $c$ are at most the number of divisors of $ad \pm 1$, which is $ \ll (ad \pm 1)^\epsilon \ll N^{3\epsilon/2}$, $\forall \epsilon >0$. By the determinant condition, $a,d $ and $c$ uniquely determine  $\sigma$, so we are done.
\item If $b \geq \sqrt{N}$, then $0< t:= b-(2-G)d \leq \frac{GN}{b}\leq G \sqrt{N}$. Also, $c \leq a$ gives that there are at most $KN$ choices for $c$. Fix $c$ and $t$, and substitute $b$ in the determinant condition to get $ad -c (t+ (2-G)d) = \pm 1$, or $d\cdot (a-(2-G)c) = c t \pm 1 \leq G K N ^{3/2} \pm 1$. Therefore, the possible choices for $d$ are at most the number of divisors of $ct \pm 1$, which is $\ll (ct \pm 1)^\epsilon \ll  N^{3\epsilon/2}$, $\forall \epsilon >0$. Since $c,d$ and $t$ uniquely determine $\sigma$, we are done.
\end{itemize}

For  $\sigma = \left( \begin{matrix} a & eb \\ c & ed \end{matrix}\right) \in A_{N,2}(K)$, consider the cases
\begin{itemize}
    \item If $d \leq \sqrt{N}$, then proceed exactly as in the first case of $A_{N,1}(K)$.
    \item If $d \geq \sqrt N$, then $cd \leq N \Rightarrow c \leq \sqrt{N}$. Also, $b \leq Ga$ (cf. the definition of $\mathcal{S}_e$ in the statement of Lemma \ref{lemma P= S+ union S-}), so that we have at most $GKN$ choices for $b$. Fix $c$ and $b$, and use the determinant condition to get $ad= bc \pm 1$. Therefore, the possible choices for $d$ are at most the number of divisors of $bc \pm 1$, which is $\ll (bc \pm 1)^\epsilon \ll N^{3\epsilon/2}$, $\forall \epsilon >0$. Since $b,c$ and $d$ uniquely determine $\sigma$, we are done.  
\end{itemize}
\end{proof}

\subsection{The parametrization of the sets $\SSS_{e} (\alpha, \beta; N)$}\label{subsection 3.3} \mbox{}
\\

 For $\alpha \geq 1$, $\beta \geq g$, and $e= \pm 1$, define the following sets
\begin{align*}
 \A_1(e, \alpha, \beta; N) := \{ &(z,x,y) \in \mathbb{Z}^3_{\geq 0} \mid z\geq y \alpha, \ x \geq z \beta,  \ x>gz, \ x+ey \leq N, \\ &
 xy \equiv e \mod{2z}, \ x,y \text{ odd}, \ z \text{ even} \},   \\[1ex]
 \A_2(e,\alpha, \beta; N) := \{ &(z,x,y) \in \mathbb{Z}^3_{\geq 0} \mid z \geq y \alpha, \ x \geq z \beta,  \ x>gz, \ x+ey \leq N, \\ & xy \equiv e \mod{z}, \ z,y \text{ odd}, \ x \text{ even} \},   \\[1ex]
 \A_3(e,\alpha, \beta; N) := \{ &(z,x,y) \in \mathbb{Z}^3_{\geq 0} \mid z \geq y \alpha, \ x \geq z \beta,  \ x>gz, \ x+ey \leq N, \\ & xy \equiv e \mod{z}, \ x,z \text{ odd}, \ y \text{ even} \}.   \\
 \end{align*}
 
 Observe that in the definition of $\A_1(e,\alpha, \beta; N)$, the condition ``$x,y$ odd" is redundant, so we drop it from now on. 
 
\begin{lemma}\label{Parametrization-}
Let $\alpha \geq 1$, $\beta \geq g$, and $e= \pm 1$, and consider the map
\begin{align*}
    \Phi_e :\ &\SSS_{e} (\alpha, \beta; N) \longrightarrow \bigcup\limits_{i=1}^3 \A_i(e,\alpha, \beta; N), \quad  \Phi_e \bigg(\left( \begin{matrix} a & eb \\ c & ed \end{matrix}\right) \bigg) = (b, a, d). 
\end{align*}
This map is a bijection, and therefore
\begin{equation*}
    |\SSS_{e} (\alpha, \beta; N)|= |\A_1(e,\alpha, \beta; N)| + |\A_2(e,\alpha, \beta; N)| + |\A_3(e,\alpha, \beta; N)|+O(N).
\end{equation*}

\end{lemma}

\begin{proof}
Fix $e$. From the congruence relations in the definition of  $\SSS_{e} (\alpha, \beta; N)$ (cf. Lemma \ref{lemma P= S+ union S-}) it follows that this map is well-defined. Since $b,a$ and $d$ completely determine the elements of $\SSS_{e} (\alpha, \beta; N)$, 
 it is also one-to-one. Now take a triple $(z,x,y) \in \bigcup\limits_{i=1}^3 \A_i(e,\alpha, \beta; N)$, and define $c:= \frac{xy-e}{z}$. To show that the matrix $\sigma := \left( \begin{smallmatrix} x & ez \\ c & ey \end{smallmatrix}\right)$ belongs to $\SSS_{e} (\alpha, \beta; N)$, we need to check that $y \leq z$, that $\sigma \equiv I, A \text{ or } B \mod 2$, and that $1 \leq c \leq x$. The first inequality follows immediately from $y \alpha \leq z$ and $\alpha \geq 1$, while the congruence relations in the definition of the sets $\A_i(e,\alpha,\beta;N)$ give the congruence relations for $\sigma$. By definition of $c$  we have $c \geq 1 \Leftrightarrow xy-1>0 \Leftrightarrow xy>1$. The problematic case $x=y=1$ together with the conditions on $x,y,z$ in the definition of the $\A_i(e,\alpha,\beta;N)$ give that  $1 \leq \alpha \leq z \leq \frac{1}{g} < 2$, so that  $z=1$, which is a contradiction since all three of $x,y$ and $z$ can not have the same parity. Finally, for $c \leq  x $, if $e=+1$ we have $c \leq  x \Leftrightarrow x(y-z) \leq 1$, which is true because $y-z \leq 0$. Thus, in the case $e=+1$, the map $\Phi_e$ is a bijection. If $e=-1$, then $c \leq  x \Leftrightarrow x(y-z) \leq -1$, which is true when $y-z<0$, since $x>0$ by $x>gz$. The only case where $\sigma \not\in \SSS_e(\alpha, \beta;N)$ is when $y=z$. Then the definition of $c \in \mathbb{Z}$ and $e=-1$ imply that $y=z=1$, and $c=x+1$, so that $\sigma =  \left( \begin{smallmatrix} x & -1 \\ x+1 & -1 \end{smallmatrix}\right)$. By the condition $x+ey \leq N$ in the definition of the sets $\A_i(e,\alpha,\beta;N)$, we have that there are at most $N+1$ such preimages of $(z,x,y)$ under $\Phi_{-1}$ that do not belong to $ \SSS_e(\alpha, \beta;N)$. Thus the result follows.
\end{proof}

 \section{Computing the asymptotic behavior of the parametrization}\label{section computing asymptotic behavior}

Now we determine the asymptotic behavior of $|\A_i(e,\alpha, \beta;N)|$, $i=1,2,3$, $e = \pm 1$. For a set $\Omega \subset \mathbb{R}^2$, the number of lattice points on the modular hyperbola 
\begin{equation*}
    \mathcal{H}_{m,h}:=\{(x,y) \in \mathbb{Z}^2 \mid xy \equiv h \mod m\}
\end{equation*}
 lying inside $\Omega$ is 
\begin{equation*}
     \N_{m,h}(\Omega) := \sum\limits_{\substack{(x,y) \in \Omega \cap \mathcal{H}_{m,h}}} 1.
\end{equation*}
Denote $\N_{m,1}(\Omega)= \N_{m}(\Omega)$. Define
\begin{align}\label{regions omega}
    \Omega_m(e,\alpha,\beta;N) & := \{(x,y) \mid x \geq m\beta, \ 0 \leq y \leq m/\alpha, \ x+ey \leq N \}, 
\end{align}
so that in this notation we have 
{\small
\begin{equation}\label{eq def on A_i sets}
\begin{split}
 \A_1(e,\alpha, \beta; N) = \{ &(z,x,y) \in \mathbb{Z}^3_{\geq 0} \mid   z \text{ even}, \ (x,y) \in \Omega_z(e,\alpha,\beta;N)\cap  \mathcal{H}_{2z,e}, \ z \geq 2 \},  \\[1ex]
\A_2(e,\alpha, \beta; N) = \{ &(z,x,y) \in \mathbb{Z}^3_{\geq 0} \mid z,y \text{ odd}, \ x \text{ even}, \ (x,y) \in \Omega_z(e,\alpha,\beta;N) \cap \mathcal{H}_{z,e}, \ z\geq 1   \},   \\[1ex]
\A_3(e,\alpha, \beta; N) = \{ &(z,x,y) \in \mathbb{Z}^3_{\geq 0} \mid  x,z \text{ odd}, \ y \text{ even}, \ (x,y) \in  \Omega_z(e,\alpha,\beta;N) \cap \mathcal{H}_{z,e} , \ z\geq 1  \}.  
  \end{split}
 \end{equation}
}
 
To find asymptotics for the above sets, we will use the following Lemma:

\begin{lemma}\label{bocaustinov}
 Let $h,m$ be integers with $m \geq 2$ and $(h,m)=1$, and $\phi$ be Euler's totient function. Then
\begin{itemize}
    \item[(i)] For $I_1,I_2$ intervals, we have
\begin{equation*}
    \N_{m,h}(I_1 \times I_2)= \frac{\phi (m)}{m^2} |I_1||I_2| + O_\epsilon \bigg( m^{1/2+\epsilon} \Big(1+\frac{|I_1|}{m}\Big) \Big(1+\frac{|I_2|}{m}\Big)\bigg).
\end{equation*}
    \item[(ii)] For  every  integer $c$ and  interval $I$ with $|I|< m$, consider the linear function $f(x) =c \pm x$ such that $f(I) \subset [0,m]$. Then
 \begin{equation*}
     \N_{m,h}(\{(x,y) \mid x \in I, \ 0\leq y \leq f(x)\})= \frac{\phi (m)}{m^2} \int_I f(x) \ dx + O_\epsilon(m^{1/2+\epsilon}).
 \end{equation*}
 Moreover, the same conclusion holds if $f$ is of the form $f(x)= c \pm K x$, $K \in \mathbb{N}$, or if $c \notin \mathbb{Z}$. 
\end{itemize}
\end{lemma}
For part (i), see e.g. Proposition A3 in \cite{BZ}, or \cite{Ustinov3}. Part (ii) is an extension of Lemma 2 in  \cite{Ustinov} (see Appendix 2 below). 

We start with the case $e=+1$. Fix $z=m$ and determine what $\Omega_m(+1,\alpha,\beta;N)$ looks like (cf. (\ref{regions omega})): first draw the line $x+y=N$, and observe that $ m \beta \geq N \implies x \geq N $, so that  $\Omega_m(+1,\alpha,\beta;N) = \emptyset$. Thus $m \leq \frac{N}{\beta}$. 
 
 Since $m\beta \leq x \leq N$ and $0 \leq y \leq \frac{m}{\alpha}$, when $\frac{m}{\alpha} \leq N-m\beta \Leftrightarrow m \leq \frac{\alpha N}{\alpha \beta +1}$, then $\Omega_m(+1,\alpha,\beta;N)$ is a trapezoid with bases given by the $x$-axis and $y=\frac{m}{\alpha}$, and sides $x=m\beta$ and $y=N-x$. When $\frac{m}{\alpha} \geq N-m\beta \Leftrightarrow m \geq \frac{\alpha N}{\alpha \beta +1} $, $\Omega_m(+1,\alpha,\beta;N)$ is a triangle with sides given by the $x$-axis, $x=m\beta$ and $y=N-x$. 
 
Therefore we have (cf. (\ref{eq def on A_i sets}))
\begin{align}\label{A_1initial}
    |\A_1(+1,\alpha,\beta;N)|= \sum\limits_{\substack{1 \leq m \leq \frac{\alpha N}{\alpha \beta +1} \\ m \text{ even}}} \N_{2m}(trapezoid) + \sum\limits_{\substack{\frac{\alpha N}{\alpha \beta +1} \leq m \leq \frac{N}{\beta} \\ m \text{ even}}} \N_{2m}(triangle).
\end{align}

 For $1 \leq m \leq \frac{\alpha N}{\alpha \beta +1}$, divide the trapezoid into a rectangle and a triangle, and apply Lemma \ref{bocaustinov}, part (i) with $I_1\times I_2= [\frac{m}{\beta}, N-\frac{m}{\alpha}] \times [0,\frac{m}{\alpha}]$ (for the rectangle) and Lemma \ref{bocaustinov}, part (ii), with $I= [N-\frac{m}{\alpha},N ]$, with $|I| = \frac{m}{\alpha} < m$, and $f(x)= N-x$ (for the triangle), to get 
 
 \begin{align}\label{Ntrapezoid}
     \N_{2m} (trapezoid) = \ &\frac{\phi (2m)}{(2m)^2} \cdot \Big( N-\frac{m}{\alpha}- m\beta \Big) \cdot \frac{m}{\alpha} \ + \ \frac{\phi (2m)}{(2m)^2} \cdot \frac{m^2}{2\alpha^2} \ + \ O_\epsilon\big(m^{-\frac{1}{2}+\epsilon}N\big).
 \end{align}
 
 For $\frac{\alpha N}{\alpha \beta +1} \leq m \leq \frac{N}{\beta}$, use Lemma \ref{bocaustinov}, part (ii), with $I= [m\beta , N]$, with $|I|= N-m\beta <m \Leftrightarrow m \geq \frac{\alpha N}{\alpha \beta +1}$, and $f(x)= N-x$ to get
 
 \begin{align}\label{Ntriangle}
     \N_{2m} (triangle) = \ &\frac{\phi (2m)}{(2m)^2} \cdot \frac{(N-m\beta)^2}{2} \ +O_\epsilon\big(m^{\frac{1}{2}+\epsilon}\big).
 \end{align}
 Now use (\ref{A_1initial}), (\ref{Ntrapezoid}) and (\ref{Ntriangle}) to get 
 \begin{align}\label{A_1withN}
  \nonumber   |\A_1(+1,\alpha,\beta;N)| & = \sum\limits_{\substack{1 \leq m \leq \frac{\alpha N}{\alpha \beta +1} \\ m \text{ even}}} \Bigg( \frac{\phi (2m)}{(2m)^2} \cdot \Big( N-\frac{m}{\alpha}- m\beta \Big) \cdot \frac{m}{\alpha} \ + \ \frac{\phi (2m)}{(2m)^2} \cdot \frac{m^2}{2\alpha^2} \Bigg) \\ \nonumber
     &+ \sum\limits_{\substack{\frac{\alpha N}{\alpha \beta +1} \leq m \leq \frac{N}{\beta} \\ m \text{ even}}} \Bigg( \frac{\phi (2m)}{(2m)^2} \cdot \frac{(N-m\beta)^2}{2} \Bigg) + O_\epsilon \big(N^{\frac{3}{2}+\epsilon} \big) \\ \nonumber
    & =  \frac{N}{4\alpha}\sum\limits_{\substack{1 \leq m \leq \frac{\alpha N}{\alpha \beta +1} \\ m \text{ even}}} \frac{\phi(2m)}{m}\  - \ \frac{1+2\alpha\beta}{8\alpha^2}\sum\limits_{\substack{1 \leq m \leq \frac{\alpha N}{\alpha \beta +1} \\ m \text{ even}}} \phi(2m) \ + \ \frac{N^2}{8} \sum\limits_{\substack{\frac{\alpha N}{\alpha \beta +1} \leq m \leq \frac{N}{\beta} \\ m \text{ even}}} \frac{\phi(2m)}{m^2}\\ \nonumber
     & -\frac{\beta N}{4} \sum\limits_{\substack{\frac{\alpha N}{\alpha \beta +1} \leq m \leq \frac{N}{\beta} \\ m \text{ even}}} \frac{\phi(2m)}{m} \ + \ \frac{\beta^2}{8} \sum\limits_{\substack{\frac{\alpha N}{\alpha \beta +1} \leq m \leq \frac{N}{\beta} \\ m \text{ even}}}\phi(2m) \ +  O_\epsilon \big(N^{\frac{3}{2}+\epsilon} \big).
     \end{align}
     
     We also use the known estimates:
     
     \begin{align*}
        & \sum\limits_{\substack{1\leq m \leq N\\ m \text{ even}}}\frac{\phi(2m)}{m}  = \sum\limits_{\substack{m \leq \frac{N}{2}}} \frac{\phi(4m)}{2m}= \frac{2N}{3\zeta (2)} + O(\log^2N) \qquad \text{(Lemma 2.2 in \cite{BG})} \\
         &\sum\limits_{\substack{1\leq m \leq N\\ m \text{ even}}} \phi(2m) = \sum\limits_{m \leq \frac{N}{2}} \phi(4m) = \frac{N^2}{3\zeta(2)} + O(N \log N) \qquad  \text{(Lemma 2.2 in \cite{BG})}\\
        & \sum\limits_{\substack{1\leq m \leq N\\ m \text{ even}}}\frac{\phi(2m)}{m^2}  = \frac{2}{3\zeta(2)} \Big( \log N + \gamma -\frac{4 \log 2}{3}- \frac{\zeta'(2)}{\zeta(2)} \Big) + O(N^{-1} \log^2 N) \qquad \text{(Section 5 in \cite{BS})}, \\
        & \text{which gives, for every $s<1$,} \sum\limits_{\substack{sN\leq m \leq N\\ m \text{ even}}}\frac{\phi(2m)}{m^2} = -\frac{2 \log s}{3\zeta(2)} + O(N^{-1} \log^2 N).
     \end{align*}
     Observe that $\frac{\alpha}{\alpha\beta+1}<1$: if $\beta>1$, then it is obvious, while if $\beta \in (\frac{g}{g^2+1}, G)$, then the function $x \mapsto \frac{\alpha}{\alpha x+1}$ is decreasing, so it achieves its maximum at $\frac{g}{g^2+1}$, and the function $x \mapsto \frac{x}{1+xg/(g^2+1)}$ is increasing and tends to $1$ at infinity, while at $x=1$ it is $<1$. Using the above estimates we get
     
     \begin{equation}\label{A_1final}
     |\A_1(+1,\alpha,\beta;N)| = \frac{N^2}{12\zeta(2)}\cdot \log
     \frac{\alpha\beta+1}{\alpha\beta} + O_\epsilon(N^{\frac{3}{2}+\epsilon}).
     \end{equation}
     
     Now, for $\A_2(+1,\alpha,\beta;N)$, things get more complicated because we have additional restrictions on the parity of $x$ 
      and $y$ (cf. (\ref{eq def on A_i sets})):
      \begin{equation*}
          |\A_2(+1,\alpha,\beta;N)|= \sum\limits_{\substack{m \geq 1 \\ m \text{ odd}}}  \ \sum\limits_{\substack{(x,y) \in \Omega_m(+1,\alpha,\beta;N)\cap \mathcal{H}_{m,1}\\  x \text{ even,} \ y \text{ odd}}}1,
      \end{equation*}
     while  Lemma \ref{bocaustinov} does not take the parity of $x$ and $y$ into consideration. Therefore, we need to define the auxiliary sets $\Omega'_m(+1,\alpha,\beta;N)$: We have that $xy \equiv 1 \mod m \Leftrightarrow 2 \cdot \frac{x}{2} \cdot y \equiv 1 \mod m \Leftrightarrow \frac{x}{2} \cdot y \equiv \overline{2} \mod m$, where $\overline{2}\cdot 2 \equiv 1 \mod m$. Define
     \begin{equation}\label{eq def Omega prime}
         \Omega'_m(+1,\alpha,\beta;N):= \Big \{(x,y) \mid x \geq \frac{m \beta}{2}, \ 0 \leq y \leq \frac{m}{\alpha}, \ 2x+y \leq N \Big\},
     \end{equation}
     so that $(x,y) \in \Omega_m(+1,\alpha,\beta;N) \Leftrightarrow (\frac{x}{2},y) \in \Omega'_m(+1,\alpha,\beta;N)$. Therefore, we can write
     
   \begin{align*}
          |\A_2(+1,\alpha,\beta;N)| & = \sum\limits_{\substack{m \geq 1 \\ m \text{ odd}}} \sum\limits_{\substack{(x,y) \in \Omega'_m(+1,\alpha,\beta;N)\\ xy \equiv \overline{2} \mod m\\  y \text{ odd}}}1 \\
          & = \sum\limits_{\substack{m \geq 1 \\ m \text{ odd}}} \sum\limits_{\substack{(x,y) \in \Omega'_m(+1,\alpha,\beta;N)\\ xy \equiv \overline{2} \mod m}}1 - \sum\limits_{\substack{m \geq 1 \\ m \text{ odd}}} \sum\limits_{\substack{(x,y) \in \Omega'_m(+1,\alpha,\beta;N)\\ xy \equiv \overline{2} \mod m\\  y \text{ even}}}1.
      \end{align*}  
      
      The above sums are finite, because $\Omega'_m(+1,\alpha,\beta;N)$ can be shown to be a bounded region in the same way as $\Omega_m(+1,\alpha,\beta;N)$. Now we have $xy \equiv \overline{2} \mod m \Leftrightarrow x \cdot \frac{y}{2} \cdot 2 \equiv \overline{2} \mod m \Leftrightarrow x \cdot \frac{y}{2} \equiv \overline{2} \cdot \overline{2} \mod m$, and that $(x,y) \in \Omega'_m(+1,\alpha,\beta;N) \Leftrightarrow (x,\frac{y}{2}) \in \Omega_m(+1,2\alpha,\frac{\beta}{2};\frac{N}{2})$. Therefore
      
       \begin{align}\label{A_2initial}
          |\A_2(+1,\alpha,\beta;N)| = \sum\limits_{\substack{m \geq 1 \\ m \text{ odd}}} \sum\limits_{\substack{(x,y) \in \Omega'_m(+1,\alpha,\beta;N)\\ xy \equiv \overline{2} \mod m}}1 \ -  \ \sum\limits_{\substack{m \geq 1 \\ m \text{ odd}}} \sum\limits_{\substack{(x,y) \in \Omega_m(+1,2\alpha,\beta/2;N/2)\\ xy \equiv \overline{2} \cdot \overline{2} \mod m}}1.
      \end{align}  
     
     Now we can apply Lemma \ref{bocaustinov}: starting with $\Omega'_m(+1,\alpha,\beta;N)$ (cf. (\ref{eq def Omega prime})), we draw the line $2x+y = N$, and observe that, since $x \geq \frac{m}{\beta_2}$, $\Omega'_m(+1,\alpha,\beta;N)= \emptyset \Leftrightarrow m \geq \frac{N}{\beta}$. Thus $m \leq \frac{N}{\beta}$.  
     
     Similar to the previous case of $\Omega_m(+1,\alpha,\beta;N)$, here $\Omega'_m(\alpha,\beta;N)$ is a trapezoid of area $\big(\frac{N}{2}- \frac{m}{2\alpha}-\frac{m\beta}{2}\big) \frac{m}{\alpha} + \frac{m^2}{4\alpha^2}$ when $m \leq \frac{\alpha N}{\alpha\beta +1}$, and a triangle of area $\frac{(N-m\beta)^2}{4}$ when $\frac{\alpha N}{\alpha\beta +1} \leq m \leq \frac{N}{\beta}$. Accordingly, $\Omega_m(+1,2\alpha,\frac{\beta}{2};\frac{N}{2})$ is a trapezoid of area $\big(\frac{N}{2}- \frac{m}{2\alpha}-\frac{m\beta}{2}\big) \frac{m}{2\alpha} + \frac{m^2}{8\alpha^2}$ when $m \leq \frac{\alpha N}{\alpha\beta +1}$, and a triangle of area $\frac{(N-m\beta)^2}{8}$ when $\frac{\alpha N}{\alpha\beta +1} \leq m \leq \frac{N}{\beta}$.

 Observe that in (\ref{A_2initial}), we have that, in the notation of Lemma \ref{bocaustinov}, $h=\overline{2}$ or $\overline{2}\cdot \overline{2}$, where $(\overline{2},m)= (\overline{2}\cdot \overline{2},m)=1$. 
 
 For the first sum in (\ref{A_2initial}): for $1 \leq m \leq \frac{\alpha N}{\alpha \beta +1}$ apply Lemma \ref{bocaustinov}, part (i), with $I_1\times I_2= [\frac{m\beta}{2}, \frac{N}{2}-\frac{m}{2\alpha}] \times [0,\frac{m}{\alpha}]$ (for the rectangle), and Lemma \ref{bocaustinov}, part (ii), with $I=[ \frac{N}{2}-\frac{m}{2 \alpha},\frac{N}{2}]$, of length $|I|= \frac{m}{2\alpha} <m$, and $f(x)= N-2x$ (for the triangle). For $\frac{\alpha N}{\alpha \beta +1} \leq m \leq \frac{N}{\beta}$ apply Lemma \ref{bocaustinov}, part (ii), with $I=[\frac{m\beta}{2}, \frac{N}{2}]$, with length $|I|= \frac{N-m\beta}{2}$, and $f(x)= N-2x$, whose maximum is $N-m\beta$. To do this, we need to prove that $|I|< m$, and $f(x)< m$, for all $x \in I$. But this follows from $m \geq \frac{\alpha N}{\alpha \beta +1}= \frac{N}{\beta + \frac{1}{\alpha}} \geq \frac{N}{\beta +1}$. Therefore, the main term in the first sum in (\ref{A_2initial}) is equal to
 
 \begin{equation}\label{firstsum}
  \sum\limits_{\substack{ 1\leq m \leq \frac{\alpha N}{\alpha \beta +1} \\ m \text{ odd}}} \frac{\phi(m)}{m^2}\Big[ \Big( \frac{N}{2}- \frac{m}{2\alpha}-\frac{m\beta}{2}\Big) \cdot \frac{m}{\alpha} + \frac{m^2}{4\alpha^2}\Big] + \sum\limits_{\substack{ \frac{\alpha N}{\alpha \beta +1} \leq m \leq \frac{N}{\beta} \\ m \text{ odd}}} \frac{\phi(m)}{m^2} \cdot \frac{(N-m \beta)^2}{4},
 \end{equation}
     with an error of $O_\epsilon(N^{3/2+\epsilon})$.
    
     For the second sum in (\ref{A_2initial}): for $1 \leq m \leq \frac{\alpha N}{\alpha \beta +1}$ apply Lemma \ref{bocaustinov}, part (i), with $I_1\times I_2= [\frac{m\beta}{2}, \frac{N}{2}-\frac{m}{2\alpha}] \times [0,\frac{m}{2\alpha}]$ (for the rectangle), and Lemma \ref{bocaustinov}, part (ii), with $I=[ \frac{N}{2}-\frac{m}{ 2\alpha},\frac{N}{2}]$, of length $|I|= \frac{m}{2\alpha} <m$, and $f(x)= N-2x$ (for the triangle). For $\frac{\alpha N}{\alpha \beta +1} \leq m \leq \frac{N}{\beta}$, apply Lemma \ref{bocaustinov}, part (ii), with $I=[\frac{m \beta}{2},\frac{N}{2}]$ and $f(x)= \frac{N}{2}-x$. Therefore, the main term in the second sum in (\ref{A_2initial}) is equal to
     
     \begin{align}\label{secondsum}
      \sum\limits_{\substack{ 1\leq m \leq \frac{\alpha N}{\alpha \beta +1} \\ m \text{ odd}}} \frac{\phi(m)}{m^2}\Big[ \Big( \frac{N}{2}- \frac{m}{2\alpha}-\frac{m\beta}{2}\Big) \cdot \frac{m}{2\alpha} + \frac{m^2}{8\alpha^2}\Big]  + \sum\limits_{\substack{ \frac{\alpha N}{\alpha \beta +1} \leq m \leq \frac{N}{\beta} \\ m \text{ odd}}} \frac{\phi(m)}{m^2} \cdot \frac{(N-m \beta)^2}{8},
     \end{align}
      with an error of $O_\epsilon(N^{3/2+\epsilon})$.

    Combining the relations (\ref{A_2initial}),(\ref{firstsum}) and (\ref{secondsum}) we get  
    \begin{align}\label{A_2middle}
  \nonumber  |\A_2(+1,\alpha,\beta;N)| & = \sum\limits_{\substack{ 1\leq m \leq \frac{\alpha N}{\alpha \beta +1} \\ m \text{ odd}}} \frac{\phi(m)}{m^2} \cdot \bigg( \frac{Nm}{4\alpha} - \frac{(1+2\alpha \beta)m^2}{8\alpha^2}\bigg) \\
 \nonumber   & + \sum\limits_{\substack{ \frac{\alpha N}{\alpha \beta +1} \leq m \leq \frac{N}{\beta} \\ m \text{ odd}}} \frac{\phi(m)}{m^2} \cdot \frac{(N-m \beta)^2}{8} \ + \ O_\epsilon(N^{\frac{3}{2}+\epsilon}) \\ \nonumber
       & = \frac{N}{4 \alpha} \sum\limits_{\substack{ 1\leq m \leq \frac{\alpha N}{\alpha \beta +1} \\ m \text{ odd}}} \frac{\phi(m)}{m} -\frac{1+2\alpha \beta}{8\alpha^2}\sum\limits_{\substack{ 1\leq m \leq \frac{\alpha N}{\alpha \beta +1} \\ m \text{ odd}}} \phi(m) + \frac{N^2}{8} \sum\limits_{\substack{ \frac{\alpha N}{\alpha \beta +1} \leq m \leq \frac{N}{\beta} \\ m \text{ odd}}} \frac{\phi(m)}{m^2} \\ \nonumber
       & - \frac{\beta N}{4} \sum\limits_{\substack{ \frac{\alpha N}{\alpha \beta +1} \leq m \leq \frac{N}{\beta} \\ m \text{ odd}}}\frac{\phi(m)}{m} + \frac{\beta^2}{8} \sum\limits_{\substack{ \frac{\alpha N}{\alpha \beta +1} \leq m \leq \frac{N}{\beta} \\ m \text{ odd}}} \phi(m) \ + \  O_\epsilon(N^{\frac{3}{2}+\epsilon}). 
    \end{align}  
 
 The above equality together with the estimates 
 \begin{align*}
          \sum\limits_{\substack{1 \leq m \leq N \\ m \text{ odd}}} \phi(m) & = \frac{N^2}{3\zeta(2)} + O(N \log N) \qquad \text{(Lemma 2.2 in \cite{BG})} \\
          \sum\limits_{\substack{1 \leq m \leq N \\ m \text{ odd}}} \frac{\phi(m)}{m} & = \frac{2N}{3\zeta(2)} + O(N \log N) \qquad \text{(Lemma 2.2 in \cite{BG})} \\
          \sum\limits_{\substack{1 \leq m \leq N \\ m \text{ odd}}} \frac{\phi(m)}{m^2} & = \frac{2}{3\zeta(2)} \bigg( \log N + \gamma + \frac{2 \log 2}{3} -\frac{\zeta'(2)}{\zeta(2)} \bigg)  +O(N^{-1} \log^2 N) \qquad \text{(Section 5 in \cite{BS})},
      \end{align*}
 
 gives that 
 \begin{equation}\label{A_2final}
     |\A_2(+1,\alpha,\beta;N)| = \frac{N^2}{12\zeta(2)} \cdot \log \frac{\alpha \beta +1}{\alpha \beta}  \ + \  O_\epsilon(N^{\frac{3}{2}+\epsilon}).
 \end{equation}

 Finally, for $\A_3(+1,\alpha,\beta;N)$ (cf. (\ref{eq def on A_i sets})), set $\Omega''_m(+1,\alpha, \beta ;N):= \{(x,y) \mid x \geq m\beta, \ 0 \leq y \leq \frac{m}{2\alpha}, \ x+2y \leq N\}$. Straightforward calculations give that

 \begin{equation}\label{A_3final}
     |\A_3(+1,\alpha,\beta;N)|=  \frac{N^2}{12\zeta(2)} \cdot\log \frac{\alpha \beta +1}{\alpha \beta}  \ + \  O_\epsilon(N^{\frac{3}{2}+\epsilon}).
 \end{equation}
 
  Now, Lemma \ref{Parametrization-} together with (\ref{A_1final}), (\ref{A_2final}) and (\ref{A_3final}) give that
  
  \begin{equation}\label{S+estimate}
      |\SSS_{+1} (\alpha, \beta; N)|= \frac{N^2}{4\zeta(2)} \cdot \log \frac{\alpha\beta+1}{\alpha \beta}  \ + \ O_\epsilon(N^{\frac{3}{2}+ \epsilon}).
  \end{equation}
 For the case $e=-1$, fix $z=m$ and observe that $\Omega_m(-1,\alpha,\beta;N)$ is a trapezoid of area $(N-m\beta )  \frac{m}{\alpha}+ + \frac{m^2}{2\alpha^2}$ when $m \leq \frac{N}{\beta}$, a triangle of area $\frac{1}{2}\big(N+ \frac{m}{\alpha}-m \beta \big)^2$ when $\frac{N}{\beta} \leq m \leq \frac{\alpha N}{\alpha \beta -1}$, and is empty when $m\geq \frac{\alpha N}{\alpha\beta-1}$.

 Using Lemma \ref{bocaustinov} and the number theoretical estimates as before, we get 
 
 \begin{align}\label{A1(-1)}
 |\A_1(-1,\alpha, \beta;N)| & = \frac{N^2}{12 \zeta(2)} \cdot  \log \frac{\alpha \beta}{\alpha \beta -1} \ +O_\epsilon(N^{\frac{3}{2}+ \epsilon}). 
 \end{align}

For $i=2,3$, we proceed in the same way as we did for the case $e=+1$ to get 
\begin{equation}\label{A2and3(-1)}
   |\A_i(-1,\alpha, \beta;N)|=  \frac{N^2}{12 \zeta(2)} \cdot  \log \frac{\alpha \beta}{\alpha \beta -1} \ +O_\epsilon(N^{\frac{3}{2}+ \epsilon}). 
\end{equation}
 Combining Lemma \ref{Parametrization-} together with (\ref{A1(-1)}) and (\ref{A2and3(-1)}) we get that 
\begin{equation}\label{S(-1)}
|\SSS_{-1}(\alpha, \beta ; N)| = \frac{N^2}{4 \zeta(2)} \log \frac{\alpha \beta}{\alpha \beta -1}  \ + O_\epsilon(N^{\frac{3}{2}+ \epsilon}).
\end{equation}
 
 Finally, combining Lemmas \ref{lemmaW=S+S} and \ref{lemmaAasymptotics} with (\ref{S(-1)}) and (\ref{S+estimate}) and taking $N=\lfloor e^{\frac{R}{2}}\rfloor$, we get Theorem \ref{thm equidist of OCF QI}.

 \section{Appendix 1}\label{Appendix various reduc}
 In this section we will illustrate the difference between various types of reduced QIs. If $*$ is any of the letters in \{R,E,O,B\}, then a QI $\omega$ is called a $*$-reduced QI if it has a purely periodic $*$CF-expansion. It has been shown \cite{BM, Daj4, Ono, Pa} that 
 \begin{proposition}\label{proposition defn of R's} A QI $\omega >1$ is a
 \begin{itemize}
 \item[(i)]  R-reduced (or regular reduced) QI if and only if $\omega^* \in [-1,0]$.
 \item[(ii)]  E-reduced  QI if and only if $\omega^* \in [-1,1]$.
 \item[(iii)] B-reduced  QI if and only if $\omega^* \in [0,1]$.
 \item[(iv)] O-reduced QI if and only if $\omega^* \in [-G, 2-G]$.
 \end{itemize}
 \end{proposition} 
 In what follows we denote 
 \begin{equation*}
       [ \, (a_1,e_1), (a_2,e_2), (a_3,e_3),\ldots \,]= a_1 + \cfrac{e_1}{a_2+ \cfrac{e_2}{a_3 + \ddots}},
 \end{equation*}
 where $a_i \in \mathbb{N}$. 
 \begin{Ex1*} Let $\omega = [\,\overline{(a,+1)}\,]= a+\frac{1}{\omega}$, where $a \geq1$. The minimal polynomial of $\omega$ is $\omega^2 - a \omega -1$, so that
 \begin{equation*}
     \omega = \frac{a+ \sqrt{a^2+4}}{2}>1 \quad \text{and } \quad \omega^*= \frac{a- \sqrt{a^2+4}}{2}.
 \end{equation*}
 Observe that for all $\alpha \geq 1$ we have $-1 < \omega^* < 0$. Therefore, such a QI $\omega$ is regular-, E-, and O-reduced, but not B-reduced. When $a$ is odd, then the OCF-expansion of $\omega$ is $[\,\overline{(a,+1)}\,]$. By Remark \ref{remark finding eta}, the largest eigenvalue of $\widetilde{\Omega}(\omega) = \left( \begin{smallmatrix} a & 1 \\ 1 & 0 \end{smallmatrix}\right)^2=\left( \begin{smallmatrix} a^2+1 & a \\ a & 1 \end{smallmatrix}\right) $ is $a \omega +1= \omega^2$, so that $\varrho_o(\omega)= 4  \log \omega$. When $a$ is even, we use the insertion algorithm based on the identity 
 \begin{equation}\label{insertion identity}
     A+\cfrac{\varepsilon}{B+\xi}= A+\varepsilon+\cfrac{-\varepsilon}{1+\cfrac{1}{B-1+\xi}},
 \end{equation}
 where $A,B \in \mathbb{N}$, $\xi \in [-1,1]$, and $\varepsilon = \pm1$, to obtain the OCF-expansion of $\omega$, which turns out to be $[\,\overline{(a+1,-1),(1,+1),(a-1,+1)}\,]$ (see \cite{HartKra}, also 
 \cite{Mas}). This expansion is purely periodic, as expected. Moreover, the largest eigenvalue of $\widetilde{\Omega}(\omega)= \left( \begin{smallmatrix} a+1 & -1 \\ 1 & 0 \end{smallmatrix}\right)\left( \begin{smallmatrix} 1 & 1 \\ 1 & 0 \end{smallmatrix}\right)\left( \begin{smallmatrix} a-1 & 1 \\ 1 & 0 \end{smallmatrix}\right)=  \left( \begin{smallmatrix} a^2+1 & a \\ a & 1 \end{smallmatrix}\right)$ is $a\omega +1=\omega^2$, so that $\varrho_o(\omega)= 4 \log \omega$. 
 \end{Ex1*}
 
 \begin{Ex2*} Let
$\omega =[\, \overline{(a,-1)}\,]=a-\frac{1}{\omega}$, where $a \geq 3$. The minimal polynomial of $\omega$ is $\omega^2-a\omega +1$, so that $\text{disc}(\omega)>0 \Leftrightarrow a \geq 3$, and
\begin{equation*}
     \omega = \frac{a+ \sqrt{a^2-4}}{2}>1 \quad \text{and } \quad \omega^*= \frac{a- \sqrt{a^2-4}}{2}.
\end{equation*}
Observe that for all $a \geq 3$ we have $0< \omega^* \leq 2-G$. Therefore, such a QI $\omega$ is O-, E-, and B-reduced, but not regular reduced. When $a$ is odd, the OCF-expansion of $\omega$ is $[\, \overline{(a,-1)}\,]$. The largest eigenvalue of $\widetilde{\Omega}(\omega)=\left( \begin{smallmatrix} a & -1 \\ 1 & 0 \end{smallmatrix}\right)$ is $\omega$, so that $\varrho_o(\omega)= 2 \log \omega$. When $a$ is even, in a similar way as in Example 1, the OCF-expansion of $\omega$ is $[\,\overline{(a-1,+1),(1,+1),(a-1,-1)} \,]$, so that the largest eigenvalue of $\widetilde{\Omega}(\omega)= \left( \begin{smallmatrix} a^2-1 & -a \\ a & -1 \end{smallmatrix}\right)$, is $a\omega -1=\omega^2$, so that $\varrho_o(\omega)= 4 \log \omega$. Finally, using (\ref{insertion identity}), we obtain the RCF-expansion of $\omega$. As expected, it is not purely periodic: $[\,(a-1,+1),(1,+1),\overline{(a-2,+1),(1,+1)} \,]$.
\end{Ex2*}

\begin{Ex3*} Let $\omega = [\,\overline{(a_1,+1),(a_2,-1)} \,]= a_1+ \cfrac{1}{a_2-\frac{1}{\omega}}$, where $a_1\geq1, \ a_2 \geq 3$. Then $\omega$ is the largest root of the polynomial $a_2\omega^2-(2+a_1a_2)\omega +a_1$:
\begin{equation*}
    \omega= \
    \frac{2+a_1a_2+ \sqrt{4+ (a_1a_2)^2}}{2a_2} >1, \quad 
    \text{and } \quad \omega^*= \frac{2+a_1a_2-\sqrt{4+ (a_1a_2)^2}}{2a_2}.
\end{equation*}
Observe that for all $a_1 \geq 1, \ a_2 \geq 3$ we have $0< \omega^* < 2-G$ . Therefore, such a QI $\omega$ is O-, E-, and B-reduced, but not regular reduced. If both $a_i$'s are odd, the largest eigenvalue of $\widetilde{\Omega}(\omega)=\left( \begin{smallmatrix} (a_1a_2+1)^2-a_1a_2 & -a_1^2a_2 \\ a_1a_2^2 & 1-a_1a_2 \end{smallmatrix}\right) $ is $a_1a_2^2\omega +1-a_1a_2$, so that $\varrho_o(\omega)= 2 \log (a_1a_2^2\omega +1-a_1a_2) $.  If both $a_i$'s are even, the OCF-expansion of $\omega$ is $[\, \overline{(a_1+1,-1),(1,+1),(a_2-1,-1)}\,]$, and one can compute the length of $\omega$ as before. Similarly for the remaining two cases where the $a_i$'s do not have the same parity. 
\end{Ex3*}
\begin{Ex4*} Let $\omega= [\, \overline{(a,1),(2,-1)} \,]= a+\cfrac{1}{2-\frac{1}{\omega}}$, where $a \geq 3$. Then 
\begin{equation*}
    \omega= \frac{a+1+ \sqrt{a^2+1}}{2}>1, \quad \text{and } \quad \omega^*=\frac{a+1- \sqrt{a^2+1}}{2}.
\end{equation*}
Observe that for $a >3 $ we have  $2-G < \omega^* < 1$, so that $\omega$ is an E- and B- reduced QI, but not an O- or regular QI. 
\end{Ex4*}

 It is interesting to observe that if $\omega \in \R_* \cap \R_{*'}$, where $*$ and $*'$ are different elements of $\{R, E, O ,B\}$, then we do not necessarily have $\varrho_*(\omega)=\varrho_{*'}(\omega)$. For example, consider 
 \begin{equation*}\label{example omega appendix}
     \omega=[\, \overline{(a,1),(b,1)}\,], \qquad \text{where } a \text{ is even and } b \text{ is odd.}
 \end{equation*}
Then $\omega$ has a purely periodic RCF expansion, and so is a regular reduced QI. To find its length $\varrho(\omega)$ as a regular reduced QI, observe that $\omega$ is a fixed point of
\begin{equation*}
    \gamma:= \left( \begin{matrix} a & 1 \\ 1 & 0 \end{matrix}\right)\left( \begin{matrix} b & 1 \\ 1 & 0 \end{matrix}\right)=\left( \begin{matrix} ab+1 & a \\ b & 1 \end{matrix}\right) \notin \Gamma \qquad \text{ (cf. (\ref{defn of Gamma}))}.
\end{equation*}
Since $\det(\gamma)=1$, we have that $\varrho(\omega)=2 \log (\mathfrak{r}_\gamma)=2 \log(b\omega+1)$. 

If $b=1$, then $\omega$ is also a fixed point of 
\begin{equation*}
    \gamma':= \left( \begin{matrix} a+1 & -1 \\ 1 & 0 \end{matrix}\right) \left( \begin{matrix} a+1 & 1 \\ 1 & 0 \end{matrix}\right)\left( \begin{matrix} 1 & 1 \\ 1 & 0 \end{matrix}\right) \in \Gamma,
\end{equation*}
so that the OCF expansion of $\omega$ is $\omega =[\, \overline{(a+1,-1),(a+1,1),(1,1)} \,]$. Thus $\omega$ is an O-reduced QI as well. Observe that  $\gamma^2 = \gamma'$, so that 
 $\varrho_o(\omega)= 2 \log (\gamma')=2 \varrho(\omega)$.
 
 If $b\geq 3$, then $\omega$ is a fixed point of 
 \begin{equation*}
    \gamma'':= \left( \begin{matrix} a+1 & -1 \\ 1 & 0 \end{matrix}\right)\left( \begin{matrix} 1 & 1 \\ 1 & 0 \end{matrix}\right) \left( \begin{matrix} b & -1 \\ 1 & 0 \end{matrix}\right)\left( \begin{matrix} 1 & 1 \\ 1 & 0 \end{matrix}\right)\left( \begin{matrix} a-1 & 1 \\ 1 & 0 \end{matrix}\right)\left( \begin{matrix} b & 1 \\ 1 & 0 \end{matrix}\right) \in \Gamma.
\end{equation*}
Thus $\omega$ is an O-reduced QI, and since $\gamma^2=\gamma''$, it follows again that $\varrho_o(\omega)=2\varrho(\omega)$.

More generally, denote by $\R$ the set  of regular reduced QIs which are greater than 1. By Proposition \ref{proposition defn of R's} we have that  $\R \subset \R_O$. Then (\ref{RCF equid}) and Theorem \ref{thm equidist of OCF QI} give that 
\begin{equation*}
    \lim\limits_{R \rightarrow \infty} \ \frac{{ \big| \big\{ \omega \in \R \  \mid 
     \varrho(\omega) \leq R\big\} \big |}}{  {\big| \big\{ \omega \in \R \  \mid 
     \varrho_o(\omega)  \leq R \big\} \big |}}  >  \lim\limits_{R \rightarrow \infty} \ \frac{{ \big| \big\{ \omega \in \R \  \mid 
     \varrho(\omega) \leq R\big\} \big |}}{  {\big| \big\{ \omega \in \R_o \  \mid 
     \varrho_o(\omega)  \leq R \big\} \big |}} =2.
\end{equation*}

\section{Appendix 2}\label{Appendix Ustinov}
In this section we  describe the proof of Lemma \ref{bocaustinov}, part (ii), which extends Lemma 2 in \cite{Ustinov} to the case where $K >1$ and $c \in \mathbb{R}$, rather than just $K=1$ and $c \in \mathbb{Z}$. Following the proof in \cite{Ustinov}, let $P_1,P_2$ be such that $0 < P_2-P_1 \leq q$, where we use $q=m$ in the notation of Lemma \ref{bocaustinov}. 
Set $\mathcal{T}:= \{(x,y) \mid P_1<x\leq P_2, \ 0<y \leq f(x)\}$, and let $\mathbbm{1}_{\mathcal{T}}$ be the characteristic function of $\mathcal{T}$. Partitioning $\mathcal{T}$ into smaller pieces and translating by $q\mathbb{Z}^2$, we can assume that $\mathcal{T} \subset (-\frac{q}{2},\frac{q}{2}]^2$. We can also assume that $c \ll q$. Consider the case $K \geq 0$, the case $K<0$ being similar. Set $e(\theta):= e^{2 \pi i \theta}$, and expand $\mathbbm{1}_{\mathcal{T}}$ in a Fourier series in $\mathbb{Z}_q$:
\begin{equation*}
    F(x,y)= \sum\limits_{-\frac{q}{2}<m,n \leq \frac{q}{2}} \widehat{F}(m,n)e\Big(\frac{mx+ny}{q}\Big), 
\end{equation*}
with
\begin{equation*}\label{Fourier coefficients}
    \widehat{F}(m,n)=\frac{1}{q^2}\sum\limits_{-\frac{q}{2}<x,y \leq \frac{q}{2}}\mathbbm{1}_{\mathcal{T}}(x,y)\ e\Big(-\frac{mx+ny}{q}\Big).
\end{equation*}
Denoting $\delta_q(x)= \begin{cases}
1, & \text{if } q \mid x\\
0, & \text{if } q \nmid x,
\end{cases}$ and  $K_{q,h}(m,n):=  \sum\limits_{\substack{x,y \in \mathbb{Z}_q\\ xy \equiv h \mod q}}e(\frac{mx+ny}{q}) $ we can write
\begin{align}\label{klooster}
 \nonumber   \mathcal{N}_{q,h}(\mathcal{T}) & = \sum\limits_{-\frac{q}{2}<x,y \leq \frac{q}{2}}F(x,y) \delta_q(xy-h) = \sum\limits_{-\frac{q}{2}<x,y,m,n \leq \frac{q}{2}}\widehat{F}(x,y) \delta_q(xy-h)e\Big(\frac{mx+ny}{q}\Big)\\
    &= \sum\limits_{-\frac{q}{2}<m,n \leq \frac{q}{2}}\widehat{F}(m,n) K_{q,h}(m,n).
\end{align}

The term $(m,n)=(0,0)$ in (\ref{klooster}) gives the main term $\frac{\phi(q)}{q^2}\text{Area}(\Omega)+O_K(1)$ in Lemma \ref{bocaustinov}: combining $(h,q)=1$ with $xy \equiv h \mod q$, we get $(x,q)=1$, so that $K_{q,h}(0,0)=\phi(q)$. Moreover, $ \widehat{F}(0,0)= \frac{1}{q^2} |\mathbb{Z}^2 \cap \mathcal{T}|= \frac{1}{q^2}\big( \text{Area}(\Omega)+O_K(q)\big)$. 

It remains to show that the sum of the rest of the terms is $O_{K,\epsilon}(q^{1/2+\epsilon})$. For this we use classical bounds on Kloosterman sums: set $K_q(m,n):= K_{q,1}(m,n)$. Since $(h,q)=1$, we have $K_{q,h}(m,n)= K_q(m,nh)=K_q(mh,n)$. The Kloosterman sum $K_q(m,n)$ satisfies the Weil bound  \cite{Est}
\begin{equation*}\label{Weil bound}
    |K_q(m,n)| \leq  \sigma_0(q)q^{1/2}(q,m,n)^{1/2},
\end{equation*}
where $\sigma_0(q)$ is the number of divisors of $q$. Thus
\begin{equation*}\label{Weil bound h}
   |K_{q,h}(m,n)| = |K_q(m,nh)| \leq \sigma_0(q)q^{1/2}(q,m,nh)^{1/2}= \sigma_0(q)q^{1/2}(q,m,n)^{1/2},  
\end{equation*}
since $(h,q)=1$. Hence, it remains to show that 
\begin{equation}\label{remains to show}
    \sum\limits_{\substack{-\frac{q}{2}<m,n \leq \frac{q}{2}\\ (m,n)\neq (0,0)}}\widehat{F}(m,n) (q,m,n)^{1/2} = O_{K, \epsilon}(q^\epsilon).
\end{equation}
The first step is to obtain an alternate expression for $\widehat{F}(m,n)$:
\begin{align}
     \widehat{F}(m,n) & = \frac{1}{q^2} \sum\limits_{P_1<x \leq P_2}e\Big(-\frac{m x}{q}\Big) \sum\limits_{0< y \leq c+Kx}e\Big(-\frac{n y}{q}\Big) \label{first line of Fourier coef}\\
    \nonumber & = \frac{1}{q^2} \sum\limits_{P_1<x \leq P_2}e\Big(-\frac{m x}{q}\Big) \frac{e(-\frac{n}{q})}{1-e(-\frac{n}{q})} \bigg( 1- e\Big(-\frac{n}{q} ( c  +Kx)\Big)\bigg) \\
     & = \frac{1}{q^2} \frac{e(-\frac{n}{q})}{e(-\frac{n}{q})-1}\bigg( e\Big(-\frac{n  c }{q}\Big) \sum\limits_{P_1<x\leq P_2} e\Big(-\frac{(m+n K) x}{q}\Big) - \sum\limits_{P_1<x\leq P_2}e\Big(-\frac{m x}{q}\Big)\bigg), \label{last line Fourier coef}
\end{align}
where the last two lines are valid only if $n \neq 0$. Thus there are  four cases for $(m,n)$: $m=0$ and $n \neq 0$, $n=0$ and $m \neq 0$, $mn \neq 0$ and $q \mid m+n K$, and finally $mn \neq 0$ and $q \nmid m+nK$. We also use $|\sin x| \geq \frac{2}{\pi}|x|$, for all $x \in [-\frac{\pi}{2},\frac{\pi}{2}]$, to obtain the estimate
\begin{equation}\label{inequality coming from sine}
    \Bigg| \frac{e(-\frac{n}{q})}{e(-\frac{n}{q})-1} \Bigg| = \frac{1}{2|\sin(\frac{\pi n}{q})|} \leq \frac{q}{|n|},
\end{equation}
since $\big|\frac{n}{q}\big|\leq \frac{1}{2}$.

\begin{itemize}
    \item Case $m=0$ and $n \neq 0$: combining (\ref{last line Fourier coef}),  (\ref{inequality coming from sine}), the triangle inequality, and $P_2-P_1 \leq q$ we get 
    \begin{equation*}
        |\widehat{F}(0,n)| \leq \frac{1}{q^2}\cdot \frac{q}{|n|} \cdot 2(P_2-P_1) < \frac{1}{|n|}.
    \end{equation*}
    Therefore, the left hand side of  (\ref{remains to show})  becomes
    \begin{equation*}\label{eq bound case 1}
       \sum\limits_{\substack{-\frac{q}{2}<n \leq \frac{q}{2}\\ n\neq 0}} |\widehat{F}(0,n)| (q,n)^{1/2} \ll \sum\limits_{0<n \leq q} \frac{(q,n)^{1/2}}{n} \leq \sum\limits_{d \mid q} \sum\limits_{n_1 \leq q} \frac{d^{1/2}}{dn_1} \ll \sigma_0(q) \log q \ll_\epsilon q^\epsilon.
    \end{equation*}
 \item    Case $n=0$ and $m \neq 0$: combine (\ref{first line of Fourier coef}) with $c \ll q $ and (\ref{inequality coming from sine}) to get
 \begin{align*}
      \widehat{F}(m,0) & = \frac{1}{q^2} \sum\limits_{P_1<x \leq P_2}e\Big(-\frac{m x}{q}\Big) \lfloor c+ Kx \rfloor \\
     & = \frac{c}{q^2} \sum\limits_{P_1<x \leq P_2}e\Big(-\frac{m x}{q}\Big) + \frac{K}{q^2} \sum\limits_{P_1<x \leq P_2} xe\Big(-\frac{m x}{q}\Big) +O\bigg(\frac{1}{q^2}\sum\limits_{P_1<x \leq P_2}e \Big(-\frac{m x}{q}\Big)\bigg) \\
     & \ll \frac{1}{q|m|}+ \frac{K}{q^2} \sum\limits_{P_1<x \leq P_2} xe\Big(-\frac{m x}{q}\Big) +O\bigg(\frac{1}{q^2|m|}\bigg).
 \end{align*}
 To estimate the last sum, we can assume that  $0\leq P_1<P_2 \leq q$. Consider the derivative of the geometric series for $z=e(-\frac{m}{q})\neq 1$ and combine with (\ref{inequality coming from sine}) to get
 \begin{align*}
     \sum\limits_{P_1 < x \leq P_2} xz^x &= \frac{P_1z^{P_1+1}-P_2z^{P_2+1}}{1-z} + \frac{1}{(1-z)^2}(z^{P_1+1}-z^{P_2+1})\\
     &\ll \frac{q}{|1-z|} +\frac{1}{|1-z|^2} \ll q \frac{q}{|m|} +\bigg(\frac{q}{|m|} \bigg)^2 \ll \frac{q^2}{|m|},
 \end{align*}
 so that $|\widehat{F}(m,0)| \ll_K \frac{1}{|m|}$. Therefore, the left hand side of  (\ref{remains to show})  becomes
\begin{equation*}\label{eq bound case 2}
    \sum\limits_{\substack{-\frac{q}{2} < m \leq \frac{q}{2}\\ m \neq 0}} |\widehat{F}(m,0)|(q,m)^{1/2} \ll \sum\limits_{m \leq q} \frac{(q,m)^{1/2}}{m} \ll_{K,\epsilon} q^\epsilon.
\end{equation*}
\item  Case $mn \neq 0$ and $q \mid m+nK$: we have $ |\widehat{F}(0,n)| < \frac{1}{|n|}$ similar to the first case, so that the left hand side of  (\ref{remains to show})  becomes
\begin{equation*}\label{eq bound case 3}
     \sum\limits_{\substack{-\frac{q}{2}<n \leq \frac{q}{2}\\ n\neq 0}} |\widehat{F}(-nK,n)| (q,-nK,n)^{1/2} \ll \sum\limits_{0<n \leq q} \frac{(q,n)^{1/2}}{n}  \ll_\epsilon q^\epsilon.
\end{equation*}
\item Case $mn \neq 0$ and $q \nmid m+nK$: denote  $\|x\|:= \text{dist}(x,\mathbb{Z})$, and use (\ref{last line Fourier coef}) and
\begin{equation*}
    \sum\limits_{P_1<j \leq P_2} e\Big(-\frac{(m+n k) j}{q}\Big) \ll \frac{1}{|e(-\frac{(m+n k) j}{q})-1|} \ll \frac{1}{\|\frac{m+n k}{q}\|}
\end{equation*}
to get 
\begin{equation*}
   | \widehat{F}(m,n)| \ll \frac{1}{q^2}  \frac{q}{|n|} \bigg( \frac{1}{\|\frac{m+nK}{q}\|} + \frac{q}{|m|} \bigg).
\end{equation*}
Therefore, letting $m'=m+nK$, the left hand side of  (\ref{remains to show})  becomes
\begin{align*}
  \frac{1}{q}  \sum\limits_{\substack{0<|m|,|n| \leq \frac{q}{2}\\ m+nK \not\equiv 0 \mod q}} \frac{(q,m,n)^{1/2}}{|n| \ \|\frac{m+nK}{q}\|} & \leq \frac{1}{q}  \sum\limits_{\substack{0<|n| \leq \frac{q}{2}\\ |m'| \leq (K+1) \frac{q}{2} \\ m' \not\equiv 0 \mod q}} \frac{(q,m',n)^{1/2}}{|n| \ \|\frac{m'}{q}\|}  \leq \frac{4(K+1)}{q}   \sum\limits_{\substack{0<|n| \leq \frac{q}{2}\\ 0<|m'| \leq  \frac{q}{2} }}   \frac{(q,m',n)^{1/2}q}{|nm'| } \\
  & \ll_K \sum\limits_{d \mid q} \sum\limits_{0<|m_1|,|n_1|\leq \frac{q}{2d}} \frac{d^{1/2}}{d^2m_1n_1} \ll_{K,\epsilon} q^\epsilon \log^2q.
\end{align*}
\end{itemize}

\section{Appendix 3}\label{Appendix calculations}

In this Appendix we will prove a few simple facts about the OCF convergents which we use in the proofs of our results.

\begin{lemma}\label{lemma bounds on ratio of convergents}
For all $n\geq 1$ we have
\begin{equation*}
 e_n =+1\  \Longrightarrow \  \frac{p_{n}}{p_{n-1}} \geq g  \quad \text{ and }\quad \frac{q_{n}}{q_{n-1}} \geq g,
\end{equation*}
and
\begin{equation*}
 e_n =-1\  \Longrightarrow \  \frac{p_{n}}{p_{n-1}} \geq G+1 \quad \text{ and } \quad \frac{q_{n}}{q_{n-1}} \geq G+1.
\end{equation*}
\end{lemma}
\begin{proof}
The recurrent relations for the convergents give
\begin{align*}
    \frac{p_n}{p_{n-1}}= a_n+\cfrac{e_{n-1}}{\cfrac{p_{n-1}}{p_{n-2}}}= \cdots = a_n+\cfrac{e_{n-1}}{a_{n-1}+\cfrac{e_{n-2}}{\ddots + \cfrac{e_1}{a_1}}}.
\end{align*}
When   $e_n=+1$ we use  $a_i+e_i \geq 2$ and, following \cite{Ri2}, we get
\begin{align*}
    \frac{p_{n}}{p_{n-1}} & \geq 1+\cfrac{e_{n-1}}{a_{n-1}+\cfrac{e_{n-2}}{\ddots + \cfrac{e_1}{a_1}}} \geq  1- \cfrac{1}{a_{n-1}+\cfrac{e_{n-2}}{\ddots + \cfrac{e_1}{a_1}}} \geq  1- \cfrac{1}{3+\cfrac{e_{n-2}}{\ddots + \cfrac{e_1}{a_1}}} \geq  1- \cfrac{1}{3-\cfrac{1}{\ddots - \cfrac{1}{3}}} \\
    & \geq  1- \cfrac{1}{3-\cfrac{1}{\ddots - \cfrac{1}{3-\cfrac{1}{\ddots}}}} = g.
\end{align*}

When   $e_n=-1$ we use  $a_i+e_i \geq 2$ to get
\begin{align*}
    \frac{p_{n}}{p_{n-1}} & \geq   3- \cfrac{1}{3-\cfrac{1}{\ddots - \cfrac{1}{3-\cfrac{1}{\ddots}}}} = G+1.
\end{align*}
The corresponding results for $\frac{q_n}{q_{n-1}}$ follow in a similar way. 
\end{proof}
Now we claim that 
\begin{equation}\label{eq q_n greater than 2}
\text{ for all }  n\geq 3 , \  \text{ we have } q_n \geq 2 . 
\end{equation}
Indeed, using  $a_i+e_i \geq 2$ we have 
\begin{equation*}
    q_3=a_3q_2+e_2q_1=a_3a_2+e_2= \begin{cases}
    a_3a_2+1 \geq 2, & \text{if } e_2=+1\\
    a_3a_2-1 \geq 3a_3-1 \geq 2, & \text{if } e_2=-1,
    \end{cases}
\end{equation*}
and
\begin{align*}
    q_4=a_4q_3+e_3a_2=\begin{cases}
    a_4q_3+a_2 \geq 2, & \text{if } e_3=+1\\
    a_4q_3-a_2=a_4(a_3a_2+e_2)-a_2= (a_4a_3-1)a_2+e_2a_4,  & \text{if } e_3=-1.
    \end{cases}
\end{align*}
If $e_2=+1$ we immediately get $(a_4a_3-1)a_2+a_4 \geq 2$, while if $e_2=-1$ we get
\begin{equation*}
    (a_4a_3-1)a_2-a_4 \geq (3a_4-1)a_2-a_4= a_2(a_4-1)+a_4(a_2-1)+a_4a_2\geq 2.
\end{equation*}
Thus $q_3, q_4 \geq 2$. Assume that $q_i \geq 2$ for all $i=3,4,\ldots,n$. Then  the recurrent relations (\ref{recurrent relations}) give $q_{n+1}= a_{n+1}q_n+e_nq_{n-1}$, so that $q_n \geq 2$ when $e_n =+1$, while when $e_n=-1$, Lemma \ref{lemma bounds on ratio of convergents} gives $q_{n+1}= q_{n-1}(a_{n+1}\frac{q_n}{q_{n-1}}-1) \geq q_{n-1}(a_{n+1}(G+1)-1) \geq 2G >2 $. 

We finally claim that
\begin{equation}\label{eq p_n greater that q_n}
    \frac{p_n}{q_n} \geq1, \qquad \text{ for all } n \geq 1.
\end{equation}
It is easy to check by hand that (\ref{eq p_n greater that q_n}) holds for $n=1,2$ and $3$. Assume that (\ref{eq p_n greater that q_n}) holds for all $n=1,2,\ldots N-1$, where $N \geq 4$. Then relation (\ref{recurrent relations}) together with (\ref{eq q_n greater than 2}) give
\begin{equation*}
    p_N= \frac{p_{N-1}}{q_{N-1}}\ q_N \pm \frac{1}{q_{N-1}} \geq \frac{p_{N-1}}{q_{N-1}}\ q_N - \frac{1}{q_{N-1}} \geq \ q_N - \frac{1}{2}.
\end{equation*}
Since $p_N \in \mathbb{Z}$, it follows that $p_N \geq q_N$.

\end{document}